\documentclass[11pt]{amsart}
\usepackage{amsmath, amssymb}
\usepackage{amsfonts}
\usepackage{mathrsfs}
\usepackage[arrow,matrix,curve,cmtip,ps]{xy}
\usepackage{graphicx}
\usepackage{amscd}
\usepackage{amsthm}
\usepackage{amsfonts}
\usepackage{xypic}
\usepackage{latexsym}
\usepackage{bm}
\usepackage{stmaryrd}
\usepackage{geometry}
\usepackage{setspace}
\usepackage{url}
\usepackage{color}
\xyoption{all}

\allowdisplaybreaks

\numberwithin{equation}{section}

\newtheorem{thm}{Theorem}[section]

\newtheorem{cor}[thm]{Corollary}
\newtheorem{question}[thm]{Question}

\theoremstyle{definition}
\newtheorem{exam}[thm]{Example}
\newtheorem{defi}[thm]{Definition}
\newtheorem{remark}[thm]{Remark}
\theoremstyle{plain}

\newtheorem{lem}[thm]{Lemma}
\newtheorem{prop}[thm]{Proposition}

\newcommand{\abe}{\alpha,\beta,\epsilon}
\newcommand{\sxy}{(S;X,Y)}

\newtheorem*{theorem*}{Theorem}
\theoremstyle{remark}
\begin{document}
\title[Sufficient conditions for recognizing a 3-manifold group]{Sufficient conditions for recognizing a 3-manifold group}

\author{Karoline Null}

\address{University of Tennessee at Martin\\ 209 Hurt Street  \\ Martin, TN 38238 \\ USA}
\email{knull@utm.edu}

%\thanks{This paper is from the author's doctoral dissertation.}

\date{\today}



\begin{abstract}

In this work we ask when a group is a 3-manifold group, or more specifically, when does a group presentation come naturally from a Heegaard diagram for a 3-manifold?  We will give some conditions for partial answers to this form of the Isomorphism Problem by addressing how the presentation associated to a diagram for a splitting is related to the fundamental group of a 3-manifold.  In the process, we determine an invariant of groups (by way of group presentations) for how far such presentations are from 3-manifolds.
\end{abstract}

\maketitle

%%%%%%%%%%%%%%%%%%%%%%%%%%%%%%%%%%%%%%%%%%%%%%%%
\section{Introduction}
%%%%%%%%%%%%%%%%%%%%%%%%%%%%%%%%%%%%%%%%%%%%%%%%

Mathematicians first became interested in 3-manifolds over 100 years ago, with the writings of Henri Poincare and the continued work of Poul Heegaard.  Despite the passage of time and the attention given to them, 3-manifolds remain a very active and intriguing field for the simplest of reasons.  We still cannot answer the most basic questions about them: given 3-manifolds $M_1$ and $M_2,$ is $M_1\simeq M_2$ (the \textit{Homeomorphism Problem})? Or, is $\pi_1(M_1)\cong \pi_1(M_2)$ (the \textit{Isomorphism Problem})?

Heegaard splittings and Heegaard diagrams often provide a relatively simple means of understanding a complicated 3-manifold by transforming a 3-dimensional problem into a 2-dimensional one.  Originally Heegaard diagrams were of limited use in the study of 3-manifolds because there was (and is) not a unique diagram for a splitting of genus $\geq 2$ (see \cite{hempelbook} Exercises 2.6-2.7 for an example). Heegaard diagrams now, however, prove very useful for understanding properties of a manifold because of a wonderful correspondence between these diagrams and the fundamental group for a closed, compact 3-manifold (see \cite{S} for a proof). Transformations of a group presentation for the fundamental group thus correspond to a simple calculus of diagrams (see \cite{Zi} or \cite{S} for a survey of the subject).  However, even with the aid of diagrams, the problems for 3-manifolds are not solved because a manifold does not have a unique diagram, resulting in another problem: given two diagrams, can we determine if they represent the same manifold?  In general, the answer is no.

\begin{thm}\label{undecidable}
The problem of algorithmically deciding whether an arbitrary presentation presents a 3-manifold group is undecidable.
\end{thm}

We shall delay the proof of Theorem \ref{undecidable}. Suffice it is to say for the moment that if there existed an algorithm for deciding whether an arbitrary presentation presents a 3-manifold group, it would contradict the works of Rabin \cite{Rabin} and Perelman \cite{GP1, GP3, GP2}. The proof of that there can be no such algorithm requires the Rabin's result listed below, and a theorem by Hempel and Jaco, the proof of which required the Poincar\'{e} Conjecture (proven in \cite{GP1, GP3, GP2}).

For the following theorem, let $W$ be the set of all presentations, and $|P|$ the group presented by the presentation $P.$
\begin{thm}\label{Rabin} (Rabin, 1958) 
Let $\Pi$ be an algebraic property (i.e. a property preserved under isomorphisms) of finitely presentable groups such that (1) there exists at least one finitely presentable group which has the property $\Pi;$ (2) there exists at least one finitely presentable group  which does not have the property $\Pi$ and is not isomorphic to any subgroup of a finitely presented group having $\Pi.$ The set $S(\Pi)$
$$S(\Pi)=\{P \,:\, P\in W,\, |P| \textrm{ has property } \Pi \}$$
of all presentations (in $W$) of groups having the property $\Pi$ is not a recursive set.
\end{thm}

The proof of the following theorem depends on the Poincar\'e Conjecture.
\begin{thm}\label{JHH}(Hempel and Jaco, 1972)
A 3-manifold group is a direct product if and only if it is the direct product of a surface group and $\mathbb{Z}.$
\end{thm}

\begin{proof} Theorem~\ref{undecidable}

Suppose there exists an algorithm which decides whether an arbitrary presentation presents a 3-manifold group. Let $Q$ be a presentation for a 3-manifold group that is neither a surface group nor isomorphic to $\mathbb{Z}.$ Let $P$ be an arbitrary presentation.

If we had this supposed algorithm, then we could decide whether $Q\times P$ (which can also be written as a finite presentation) presents a 3-manifold group, and by Theorem~\ref{JHH}, $Q\times P$ is a 3-manifold group if and only if $P$ is trivial. Thus if we had an algorithm to decide whether a presentation gives a 3-manifold group, then we could decide triviality, contradicting Theorem~\ref{Rabin}. Therefore, we cannot have a general algorithm that decides if a presentation has the property that it presents a 3-manifold group because it would imply an algorithm for deciding triviality.
\end{proof}

The main contribution of Rabin's work in \cite{Rabin} is that, for any given presentation, there does not exist a general and effective method of deciding whether the group defined by the presentation has the property in question (in our case, `triviality').  However, in this paper we are not hoping to determine if any presentation gives a 3-manifold group.  Rather, we developed a test that will recognize a \textit{class} of presentations that present 3-manifold groups.  As a result, we overlook some presentations that present 3-manifold groups, but will not mistakenly identify a presentation as presenting a 3-manifold group when it does not.

Given an arbitrary diagram $D:=\sxy,$ we can construct the presentation associated to that diagram, $P(D),$ and the 3-manifold associated to that diagram, $M(D)$.  It would be natural to assume that the group presented by $P(D),$ denoted $|P(D)|,$ is isomorphic to the fundamental group of the manifold, $\pi_1(M(D)),$ but this is not always the case. In this paper we will show how the groups associated to the presentation and the 3-manifold for the same diagram are related. This is the content of Theorem~\ref{F_kcone}: $P$ naturally presents the fundamental group of a 3-manifold when $S-X$ is one component and planar.

In Section~\ref{4} we use these criteria and begin with the presentation $P$ instead of the diagram $D$. Can we determine if $P$ presents the fundamental group of a 3-manifold? This form of the Isomorphism Problem is obviously not solvable, but we prove that it is recursively enumerable in Theorem~\ref{r.e.}. Our work requires that we go from group presentations to 3-manifolds by way of diagrams. This problem is hard because the map from $P$ to $D$ is not one-to-one. As always, we consider two diagrams to be equivalent if they differ by an isotopy. Beyond this standard concession, there are two different routes we must pursue in order to get the complete class of diagrams determined by $P,$ denoted $[D(P)]$: first, for a firmly fixed $P$, there is a finite family of diagrams; and second there is also an infinite family of unreduced presentations equivalent to $P$, each of which has a finite family associated to it.

In the first instance, note that creating a diagram from a presentation is not a well-defined process because there is a choice in the order that the $Y$-curves cross as you flow around an $X$-curve, resulting in a finite equivalence class of diagrams of fixed degree. We denote the class of diagrams of fixed degree $d$  determined by $P$ as $[D(P)]_d.$ The degree $d$ diagrams are an important subset of $[D(P)],$ with a one-to-one correspondence between $[D(P)]_d$ and the class of permutation data sets determined by $P.$ The finite family can be dealt with algorithmically. The complexity of the problem is still open (that is, we just show that it can be done with brute force methods), and is detailed in \cite{thesis}.

The second route to getting the complete class of diagrams is to realize that we work from the assumption that $P$ is reduced as written because there exists a unique reduced form of each presentation. By examining the infinite class of diagrams $[D(P)],$ we are able to give partial answers to whether $P$ presents the fundamental group of a 3-manifold.  The infinite family is also important, because we have an example of an unreduced form of the presentation giving a negative answer to our question (Example~\ref{EXAMPLE1}).  Additionally, we prove that the case of a 2-generator group is completely solved in Theorem~\ref{2gSolved}.

%%%%%%%%%%%%%%%%%%%%%%%%%%%%%%%%%%%%%%%%%%%%%%%%
\section{Preliminaries}\label{2}
%%%%%%%%%%%%%%%%%%%%%%%%%%%%%%%%%%%%%%%%%%%%%%%%

\begin{defi}
Let $B^n$ denote the unit ball $\{x\in {\mathbb{R}}^n: ||x||\leq 1\},$ and $S^{n-1}$ denote the unit sphere $\{x\in {\mathbb{R}}^n: ||x||= 1\}.$ We call a space homeomorphic to $B^n$ an \textit{$n$-cell}, and a space homeomorphic to $S^{n-1}$ an \textit{$(n-1)$-sphere}.
\end{defi}

A (topological) \textit{$n$-manifold} is a separable metric space, each of whose points has an open neighborhood homeomorphic to either ${\mathbb{R}}^n$ or ${\mathbb{R}}^{n}_{+}=\{x\in {\mathbb{R}}^n: x_n\geq 0\}.$ The \textit{boundary} of an $n$-manifold $M,$ denoted $\partial M,$ is the set of points of $M$ having neighborhoods homeomorphic to ${\mathbb{R}}^{n}_{+}.$  By invariance of domain, $\partial M$ is either empty or an $n-1$ dimensional manifold and $\partial \partial M=\emptyset$ \cite{LB}. A manifold $M$ is \textit{closed} if $M$ is compact with $\partial M=\emptyset$ and the manifold is \textit{open} if $M$ has no compact component and $\partial M\neq \emptyset.$

A \textit{compression body} $V$ is a obtained from a connected surface $S$ by attaching 2-handles to $S\times\{0\}$ and capping off any 2-sphere boundary components with 3-handles. We define $$\partial_+V := S\times\{1\}$$ and $$\partial_-V=\partial V-\partial_+V,$$ the latter of which is also the result of surgery on $S\times \{0\}.$  A \textit{handlebody} is a compression body in which $\partial\_V$ is empty. Throughout this work, we assume all manifolds are oriented. Of special interest are 3-manifolds, because every compact, oriented 3-manifold has a splitting (see \cite{hempelbook} for a proof).

\begin{defi}
A \textit{$($Heegaard$)$ splitting} is a representation of a connected 3-manifold $M$ by the union of two compression bodies $V_X$ and $V_Y,$ with a homeomorphism taking $\partial_+ V_X$ to $\partial_+ V_Y.$ The resulting 3-manifold can be written $M=V_X\cup_S V_Y,$ where $S$ is the surface $\partial_+ V_X=\partial_+ V_Y$ in $M.$
\end{defi}

We call $S$ the \textit{splitting surface} and $g(S)$ the \textit{genus of the splitting.} As the genus one splittings can be effectively classified \cite{py}, we will only be considering splittings of genus $\geq 2.$  Two splittings of $M$ are \textit{isotopic} if their splitting surfaces are isotopic in $M,$ and \textit{homeomorphic} (or \textit{equivalent}) if there is a homeomorphism of $M$ taking one splitting surface to the other.

Before we formally define diagrams, we might do well to point out that one way of viewing diagrams is as a tool for splittings. Suppose we have a splitting of a 3-manifold $M=V_X\cup_S V_Y.$ A diagram shows the attaching curves for the 2-handles of $V_X$ and $V_Y$ (c.f. Figure 1). Every compact, oriented, connected 3-manifold has a splitting, and for each splitting many different curve sets could be chosen to determine the compression bodies. Thus, every 3-manifold can be studied through diagrams, a nice two-dimensional means of studying a complicated 3-dimensional object.

However, we do not need to begin with a splitting and move to the diagram. It is important to consider diagrams abstractly, so throughout this work we will begin with diagrams and determine properties of the associated 3-manifold.

\begin{defi}
A \textit{diagram} is an ordered triple $\sxy$ where $S$ is a closed, oriented, connected surface and $X:=\{X_1,\ldots,X_m\}$ and $Y:=\{Y_1,\ldots,Y_n\}$ are compact, oriented 1-manifolds in $S$ in relative general position and for which no component of $S-(X\cup Y)$ is a \textit{bigon} --- a disc whose boundary is the union of an arc in $X$ and an arc in $Y.$
\end{defi}

\begin{figure}[ht!]\label{Hdiagram}
   \begin{center}
            \includegraphics[scale=.5]{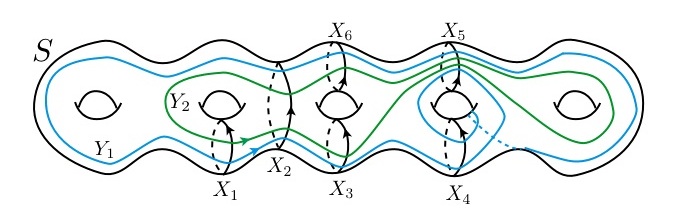}
        \caption{An example of a diagram $D=(S;X,Y)$ with $g(S)=5,$ $m=6,$ and $n=2$.}
    \end{center}
\end{figure}

This definition allows $X$ (or $Y$) to have \textit{superfluous curves}, a subset of components of $X$ (or $Y$) which could bound a planar surface in $S.$ We allow this because there is a correspondence between diagrams and presentations, under which the diagram for a 3-manifold associated to the permutations may have superfluous curves.

Two diagrams $\sxy$ and $(S^*;X^*,Y^*)$ are equivalent provided $g(S)=g(S^*)$ and there is a homeomorphism between surfaces, taking $X$ to $X^*$ and $Y$ to $Y^*.$

An \textit{arc} is a component of $Y - X$ on the surface of $S,$ and a \textit{stack} is a collection of parallel arcs in a diagram. A \textit{switchback} refers to an arc (or to a collection of parallel arcs) in a diagram which cannot be homotoped into $X,$ such that neighborhoods of both endpoints of the arc are based on the same side of the same $X$-curve (see Figure~\ref{IllustrateSwitchback}).

\begin{figure}[ht!]\label{IllustrateSwitchback}
   \begin{center}
            \includegraphics[scale=1.4]{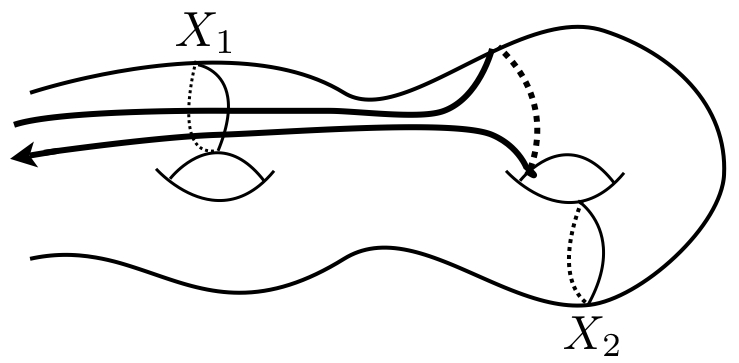}
        \caption{A switchback based on $X_1$}
    \end{center}
\end{figure}

Given a diagram, the manifold $M$ can be recovered from the diagram as follows.  For each $i=1,\ldots, m$ attach a copy of $B^2\times I$ to $S\times [0,1]$ by identifying $\partial B^2\times I$ with a neighborhood of $X_i$ in $S\times\{0\}\subset S\times[0,1].$  For each $i=1,\ldots, n$ attach a copy of $B^2\times I$ to $S\times [0,1]$ by identifying $\partial B^2\times I$ with a neighborhood of $Y_i$ in $S\times\{1\}\subset S\times[0,1].$  The resulting manifold, $M_1,$ has a 2-sphere boundary component for each planar region in $S-X$ and $S-Y.$ Obtain $M$ by attaching a copy of $B^3$ to each 2-sphere boundary component of $M_1.$  We will use this understanding of a diagram throughout this paper, viewing a diagram as giving the splitting surface sitting in a 3-manifold, with $X$ and $Y$ bounding discs on either side of $S.$

A diagram also determines a presentation.
\begin{defi}
Given a diagram $D,$ the \textit{presentation determined by $D$}, denoted $P(D),$ is a finite group presentation with one generator $x_i$ for each component $X_i\in X,$ and one relator for each component of $Y_i\in Y$, defined by recording the intersection with each $X_i$, and performing any trivial reductions. That is, each relator is obtained as $r_i:=x_{i1}^{\epsilon_1}x_{i2}^{\epsilon_2}\ldots x_{ik}^{\epsilon_k},$ where the curve $Y_i$ crosses $X_{i1},X_{i2},\ldots,X_{ik}$ in order with crossing numbers $\epsilon_i$ (see Figure 3).

When relating this to $\pi_1(M)$ we regard $x_i$ as a curve in $S$ which crosses $X_i$ with a positive crossing number and which crosses no other $X_j$. Instances of $x_ix_i^{-1}$ (or $x_i^{-1}x_i$) that appear in the presentation determined by a diagram will often not be replaced with $1$ in a relator, unless otherwise specified.
\end{defi}

\begin{figure}[ht!]
   \begin{center}
            \includegraphics[scale=.4]{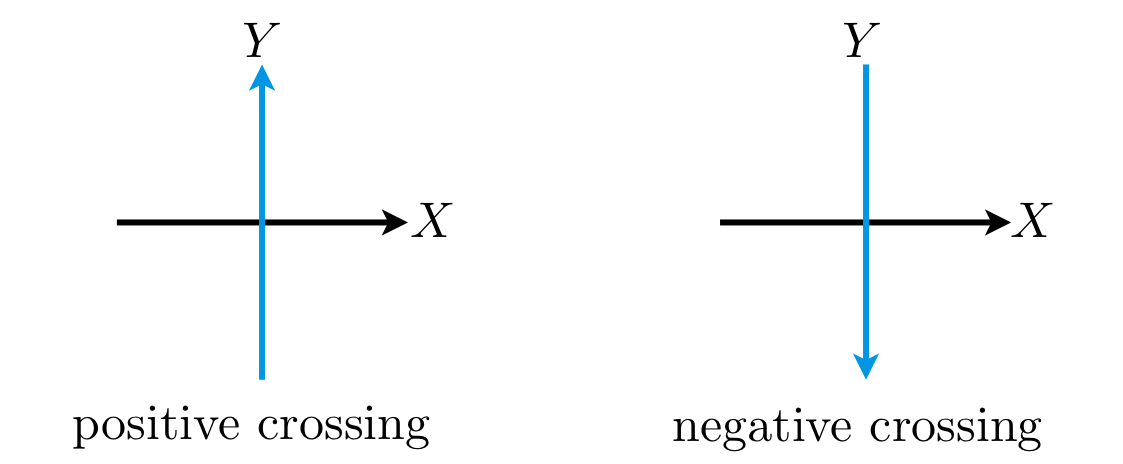}
       \caption{A positive crossing with $\epsilon = 1$ (left) and a negative crossing with $\epsilon = -1$ (right)}
    \end{center}
\end{figure}\label{epsilonA}

\begin{exam}
Consider the diagram $D=\sxy$ in Figure 1. As $X=\{X_1, X_2, X_3, X_4,$ $X_5, X_6\},$ the presentation $P(D)$ has six generators $\{x_1, x_2, x_3, x_4, x_5, x_6\}.$ As $Y=\{Y_1, Y_2\},$ $P(D)$ has two relators. We record each by flowing along the curve and recording each generator encountered with a superscript of $1$ if the crossing was positive, and $-1$ if the crossing was negative (see Figure 3). Thus

$P(D) = \langle x_1, x_2, x_3, x_4, x_5, x_6 \, :\, x_1^{-1}x_2^{-1}x_3^{-1}x_4^{-1}x_5\,x_4^{-1}x_5\,x_6\,x_2,\, \, x_1\,x_2^{-1}x_3^{-1}x_5^{-1}x_5\,x_6\,x_2\rangle.$
\end{exam}

To force the construction of $P(D)$ from $D$ to be well-defined, we set the convention that our diagram $\sxy$ is an ordered triple, where $X$ will always correspond to our generating set and $Y$ will always correspond to our set of relators. The presentation determined by a diagram is unique up to inversion and cyclic reordering.  It is well known that the group presented by the presentation, denoted $|P(D)|,$ is isomorphic to the fundamental group of the 3-manifold $M$ determined by the presentation provided $M$ is closed and $X$ is a complete meridian set (see \cite{hempelbook} for a proof). In this work, that result is a corollary to Theorem~\ref{F_kcone}.

%%%%%---Special stuff on presentations

\begin{defi}
The \textit{geometric degree} of a diagram is the number of intersection points between the two curve sets, denoted $$deg_G(D)=|X\cap Y|.$$
\end{defi}

%%-- XXXX Removing Geometric graph because removing section at end with additional checks on a graph.
%\begin{defi}
%The \textit{geometric graph} of a diagram, $G\Gamma(D),$ is the graph in the cut-open surface ($S-X$) whose vertices are the components of $X^+$ and $X^-,$ and whose edges are the $Y$-stacks, with the number of arcs in that stack recorded as the \textit{weight} of that stack.
%\end{defi}

Note that the diagram $D$ can be recovered from $G\Gamma(D)$ and some \textit{twist parameter}, which is some means of indicating how each $X_i^+$ and $X_i^-$ are identified. The twist parameter is often indicated with a point on the each of the two halves of the cut-open curve (as in Figure~\ref{example1pic}).

\begin{defi}
The \textit{algebraic degree of a presentation} $P=\langle x_1,\ldots, x_m: r_1,\ldots r_n\rangle,$ is the sum of the lengths of the relators, denoted
$$deg_A(P)= \sum_{i-1}^n |r_i|,$$ where $P$ is trivially reduced. The algebraic degree of a diagram is the algebraic degree of the presentation determined by that diagram, $deg_A(P(D)).$
\end{defi}

\begin{defi}
Given a presentation $P=\langle x_1,\ldots, x_m: r_1,\ldots r_n\rangle,$ the \textit{Whitehead graph}, $W\Gamma(P),$ is a graph with vertex set $V=\{X_i^+, X_i^-: 1\leq i\leq m\}$ and edge set determined by the relators as follows. For each $$r_i=x_{i_1}^{\epsilon_1}x_{i_2}^{\epsilon_2}\ldots x_{i_k}^{\epsilon_k},\,\,1\leq i\leq n,$$ there is an edge from $X_{i_j}^{\phi(\epsilon_j)}$ to $X_{i_j+1}^{\phi(-\epsilon_{j+1})}$ for $1\leq j\leq k$ and an edge from $X_{i_k}^{\phi(\epsilon_k)}X_{i_1}^{\phi(-\epsilon_1)},$ where the map $\phi$ is defined as $\phi(1):=+$ and $\phi(-1):=-.$  We attach a weight to each edge, equal to the number of such edges with shared endpoints, as recorded from the presentation.
\end{defi}

\begin{exam}
We construct the Whitehead graph for the Heisenberg group with presentation $$H_3:=\langle  x,y,z: [x,y]z^{-1}, [x,z], [y,z]  \rangle,$$ where $[x,z]$ denotes the commutator $xzx^{-1}z^{-1}.$ There will be six vertices: $$X^+,\,X^-,\,Y^+,\,Y^-,\,Z^+,\,Z^-.$$ To record edges, we look at each cyclically consecutive pair of generators in a relation. The first relator, $[x,y]z^{-1},$ provides us with edges between the following pairs of vertices:  $(X^+,Y^-),\,(Y^+,X^+),\,(X^-,Y^+),\,(Y^-,Z^+),\, (Z^-,X^-).$ Note that as there is no occurrence of $x^{-2},$ there will be no edge between $(X^+,X^-)$ in $W\Gamma(H_3).$ Figure~\ref{WG(H_3)} is a Whitehead graph for $H_3.$ In this case, %$W\Gamma(H_3)\cong G\Gamma(H_3)$ and
$deg_A(H_3)=deg_G(D(H_3))=13.$
 \begin{figure}[ht!]
   \begin{center}
               \includegraphics[scale=1.3]{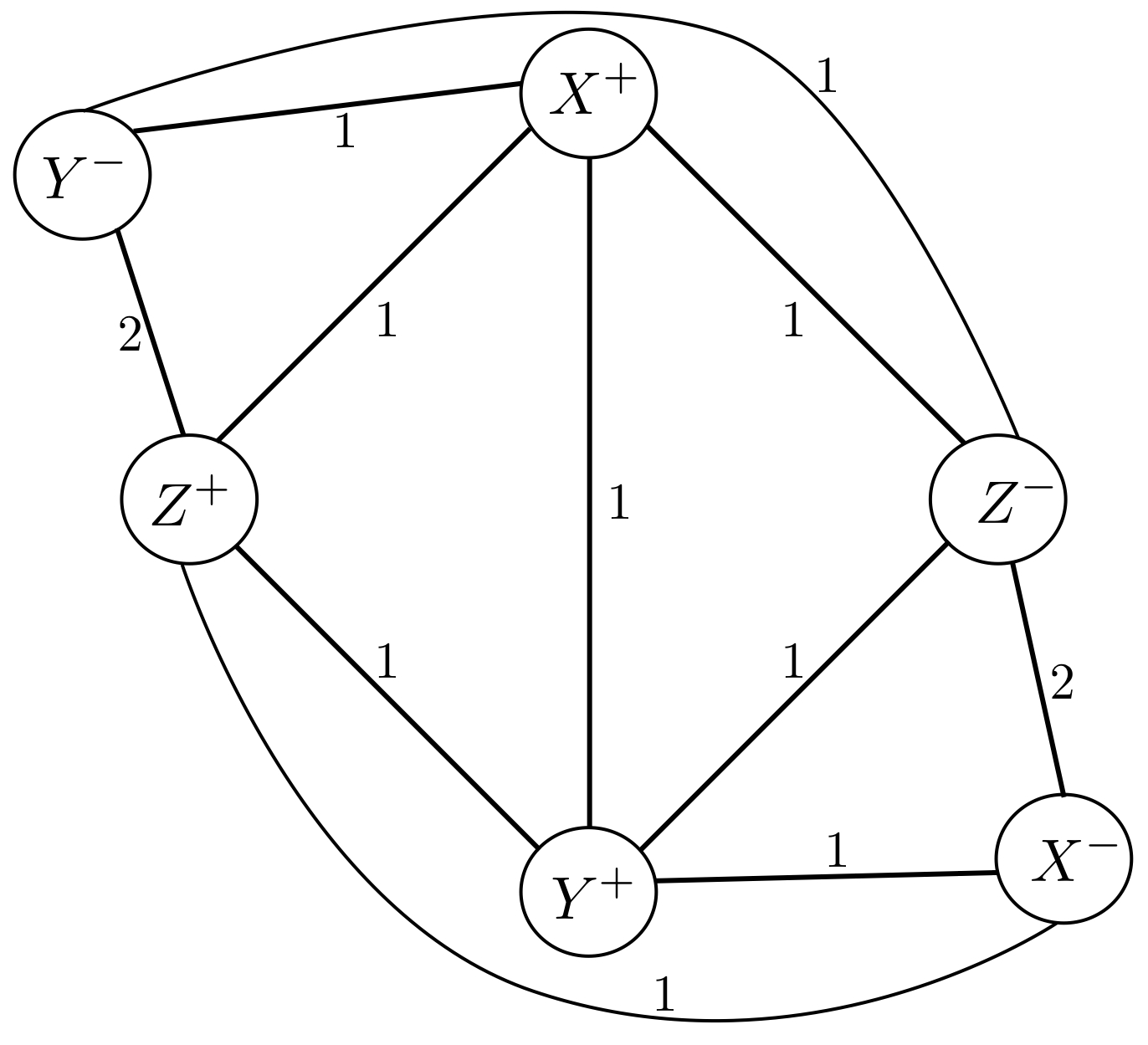}
     \caption{The graph $W\Gamma(H_3)$}\label{WG(H_3)}
    \end{center}
   \end{figure}
Note that the Whitehead graph and algebraic degree are determined from a presentation (always trivially reduced), and as such the Whitehead graph \underline{will never contain a switchback}.
\end{exam}

%%%%--- XXXX  Removing because removing info about geometric graphs.
%Let us make a few comments on Whitehead graphs. First, the Whitehead graph and algebraic degree are determined from a presentation (always trivially reduced), and as such \underline{will never contain a switchback}, whereas the geometric graph and geometric degree are determined from a diagram where switchbacks are allowed. Second, a family of diagrams determined by a presentation may have a split pair of stacks;  a Whitehead graph will have no more than one edge between any pair of vertices.  Finally, note that the graph $W\Gamma(P)$ cannot be thought of merely as the geometric graph with split pairs of stacks combined and switchbacks removed.  Removing a switchback from a diagram corresponds to a trivial reduction, creating an edge between the elements neighboring the trivial reduction in the relator. If the edge was not originally present, then the Whitehead graph contains an edge that the geometric graph did not contain. Let me end here by saying that it is possible to determine $W\Gamma(P)$ from $D(P)$ by carefully noting the new edge created when a switchback is removed (see \S\ref{typeII1} and Figures \ref{example1pic} and \ref{SWITCHBACKPIC} for an illustration).

If the diagram $D$ is given, then we have notation for the \textit{presentation determined by $D$,} denoted $P(D)$, and the \textit{manifold determined by $D,$} denoted $M(D).$ (We do not require that $P(D)$  is a reduced presentation.)

Conversely, given the presentation $P,$ we can construct a \textit{diagram determined by $P$}, denoted $D(P),$ and the \textit{manifold determined by $D(P)$}. Constructing a diagram from a presentation is roughly the reverse of constructing a presentation from a diagram (Example 2.5), but there is a bit more ambiguity. See \cite{thesis} for a complete development. The notation $D(P)$ indicates a diagram determined by $P$ and $[D(P)]$ is the class of all diagrams with presentation $P$. As mentioned at the end of Section 1, $[D(P)]$ is infinite since there is not a unique way to write a presentation.  Constructing a 3-manifold from a diagram was detailed earlier in this section.

We assume all presentations are finite and reduced unless explicitly stated otherwise.  We say that a presentation is \textit{naturally a 3-manifold presentation}, or \textit{naturally presents a 3-manifold,} if the reduced presentation determined from a diagram presents the fundamental group of the manifold determined by the diagram.

%%%%%%%%XXXXXXXXXXXXXXXXX
Consider the diagram on the cut-open surface $S-X.$ We will often consider this drawn in the plane, and ask whether the arcs embed in the plane (i.e. whether $S-X$ is a planar graph). When discussing $S-X$ in the plane, we usually disregard the twist parameter.

%%%%%%%%%%%%%%%%%%%%%%%%%%%%%%%%%%%%%%%%%%%%%%%%%%%%%%%%%%%%%%%%%%%%%%%
%%%%%%%%  CHAPTER 3

\section{Relating $P(D)$ to $\pi_1(M(D))$}\label{3}

%\begin{exam}
%Consider the diagram $\sxy$ in Figure \ref{genus2surface}. Notice $\pi_1(M(D))$ is presented by $$\langle x_1,y_1,x_2,y_2 : x_1y_1x_1^{-1}y_1^{-1}x_2y_2x_2^{-1}y_2^{-1}\rangle,$$ whereas a presentation for $P(D)$ must take the curve $Z$ into consideration:
%$$\langle x_1,y_1,x_2,y_2,z : [x_1,y_1][x_2,y_2]\rangle = \langle x_1,y_1,x_2,y_2 :[x_1,y_1][x_2,y_2]\rangle * {\mathbb{Z}}.$$
%\begin{figure}[ht!]
%   \begin{center}
%            \includegraphics[scale=.6]{PneqPi.jpg}
%        \caption{A diagram where $\pi_1(M(D))\not\simeq |P(D)|$}\label{genus2surface}
%       \end{center}
% \end{figure}

%\end{exam}

Let $V_X, V_Y$ denote the two compression bodies in the Heegaard decomposition $M=V_X\cup_S V_Y$ determined from $D$ in the standard way.  Let $\beta_0(X)$ denote the number of curves in the set $X$ (where $\beta_0$ is the $0^{th}$ Betti number).  We say \textit{$X$ is a meridian set for $V_X$} provided the curves of $X$ are homologically independent. We say  \textit{$X$ is a complete meridian set for $V_X$} if in addition $g(S)=\beta_0(X).$ If $X$ is a complete meridian set for $V_X,$ then $V_X$ is a handlebody. Recall that the cores of 2-handles are meridian discs. Let $$\partial_XM=\partial M\cap V_X=\partial\_V_X,$$ $$\partial_YM=\partial M\cap V_Y=\partial\_V_Y.$$ For brevity, we sometimes neglect to write the diagram $D$ as it is fixed throughout this section (e.g. writing $\partial_XM$ rather than $\partial_X M(D)$). We define

%and $B_1,\ldots, B_m$ denote the distinct components of $\partial_X M.$ Let $\widehat{B_i}$ denote the cone over $B_i,$ for $1\leq i \leq m.$

$$\widetilde{M}:=M\cup \textrm{ cones on components of } \partial_XM$$
 the manifold determined by the diagram $D$ together with a cone added over each boundary component.

\begin{thm}\label{F_kcone}

Given a diagram $D,$ $$|P(D)|\cong\pi_1(\widetilde{M}(D))* F_k,$$  where $\beta_0(S-X)=k+1$ and $F_k$ is the free group on $k$ generators.
\end{thm}

\begin{proof}
Assume $D$ is a diagram such that $S-X$ has $k+1$ components. Under the standard construction $M(D)=V_X\cup_S V_Y,$ with

$$V_X = S\times [0,1] \bigcup_{X_i\times\{0\}} \textrm{(2-handles)}\,\bigcup \textrm{ (3-handles)}.$$

\begin{figure}[ht!]
   \begin{center}
               \includegraphics[scale=1.3]{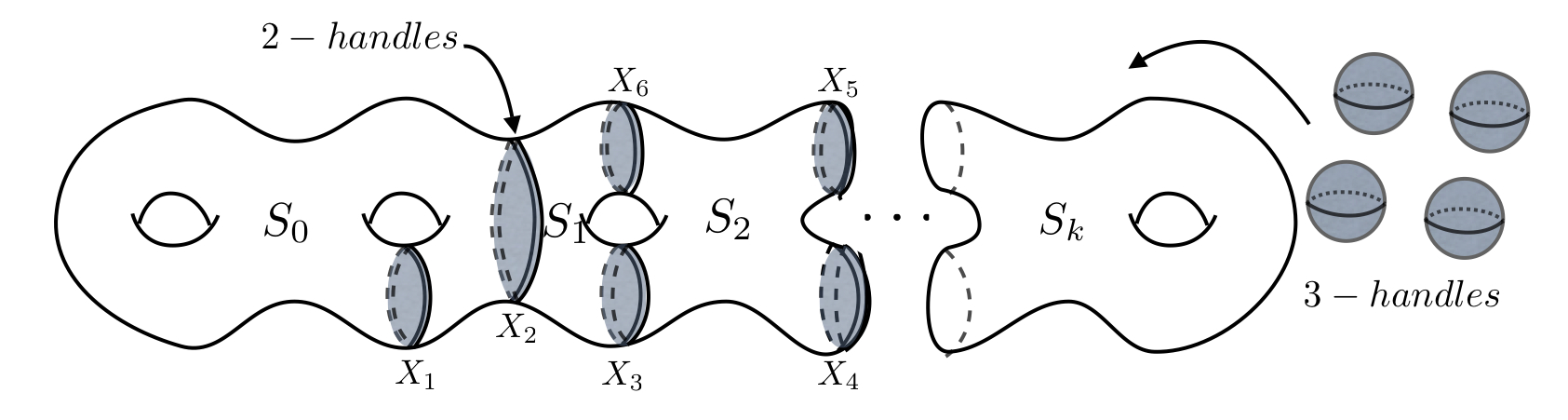}
     \caption{The standard construction of $V_X$}
    \end{center}
   \end{figure}

With $k+1$ components to $S-X,$ we can consider the co-cores of the 2-handles as determining 1-handles giving a dual handle decomposition

$$V_X := \bigcup_{i=0}^k S_i\times [0,1]\,\bigcup_{X_i\in X}\textrm{(1-handles)}\,\bigcup\textrm{ (0-handles)},$$
where $\{S_i\}$ is the set of components
$$ \partial(S\times[0,1]\bigcup_{X_i\times\{0\}}\textrm{(2-handles)})-(S\times \{1\}).$$
This is homeomorphic to $S-X$ with the boundary components capped by discs. In the following sequence, we make use of this dual handle decomposition, rather than the standard construction.

 \begin{figure}[ht!]
   \begin{center}
               \includegraphics[scale=1.3]{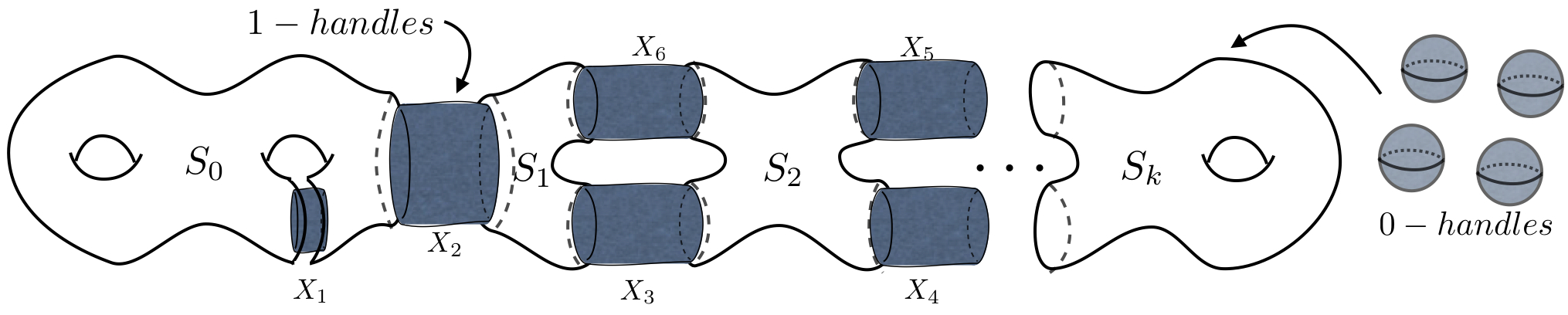}
     \caption{$V_X := \bigcup_{i=0}^k S_i\times [0,1]\,\bigcup_{X_i\in X}\textrm{(1-handles)}\,\bigcup\textrm{ (0-handles)}$}
    \end{center}
   \end{figure}

Let $$M_1:= S\times[0,1]\bigcup_{X_i\times \{0\}} \textrm{(2-handles)} \bigcup_{Y_j\times \{1\}} \textrm{(2-handles)}.$$ So $M-M_1$ is a disjoint union of 3-handles and $$\pi_1(M)\simeq \pi_1(M_1).$$  Since the cone on a 2-sphere is a 3-handle,
$$\widetilde{M}:=M\cup \textrm{(cones on components of } \partial_XM)$$ which is homeomorphic to $$M_1\bigcup_i \textrm{ (cones on }S_i),$$ where each component $S_i$ is coned to a vertex in standard position, call it $v_{S_i}.$

%We now cone $\partial_X M,$ denoting the new space $\partial_X \widehat{M}.$ Coning all of $\partial_X M$ is equivalent to the following construction. Cone each component $S_i$ of the compression body $S_i\times 0$ to a vertex in standard position, call it $v_{S_i}.$  We define this new space to be $$\widetilde{M(D)}:= M(D)\,\bigcup_{i=0}^k \partial\_\widehat{S_i}.$$

Add $k$ arcs, $\gamma_i ([0,1]) $ for $i\in\{1,\ldots,k\},$ to $\widetilde{M},$ connecting each cone point $v_{S_i}$ to $v_{S_0},$ such that

$$\gamma_i(0) = v_{S_0},\,\,\,\,\, \gamma_i(1) = v_{S_i},$$ resulting in a space which is homotopy equivalent to $$\widehat{M}:=M_1\bigcup \textrm{(cone on }\cup S_i).$$

 \begin{figure}[ht!]
   \begin{center}
               \includegraphics[scale=1.3]{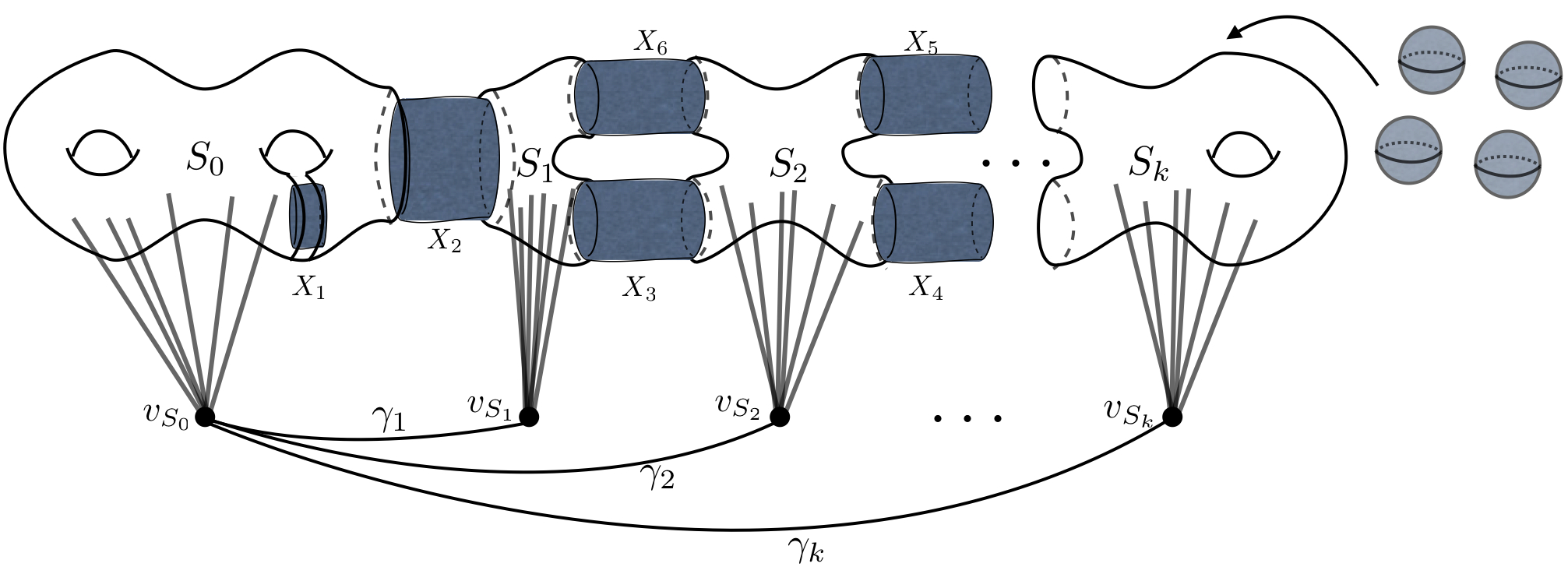}
     \caption{$\widehat{M}:=M_1\cup \textrm{(cone on }\cup S_i)$}\label{section1}
    \end{center}
   \end{figure}

%The $k$ arcs inserted into the connected space $\widetilde{M}$ are equivalent to wedging $k$ copies of $S^1$ (a circle) into the space. Thus we have

Since adding an arc (along its boundary) to a connected complex is equivalent to wedging a circle onto the complex, and this adds a free factor of $\mathbb{Z}$ to its fundamental group (an application of Seifert-van Kampen), we have

$$\pi_1(\widehat{M})\simeq\pi_1(\widetilde{M}\,\bigcup_{i=1}^k \gamma_i)\simeq\pi_1(\widetilde{M}\,\bigvee_{i=1}^k S^1)\simeq\pi_1(\widetilde{M})*F_k.$$

%In $M\cup \partial_X\widehat{M}$ there is a 1-handle, $h_{X_i}([0,1]),$ for each $X_i$ such that the initial end, $h_{X_i}(0),$ is on $\partial_+ S_i$ and the terminating end, $h_{X_i}(1)$, is on $\partial_+ S_j.$ Note $i$ may equal $j,$ as  will happen with $h_{X_1}$ on component $S_0,$ in Figure \ref{section1}.  Hence for each $X_i\in X,$ define
%\begin{align*}
%X_{i1}' &:= v_{S_0}\cup\gamma_i\cup v_{S_i}\cup (\partial\widehat{S_i} \times [0,1]),\\
%X_{i2}' &:= (\partial\widehat{S_j}\times [1,0]) \cup v_{S_j}\cup \gamma_j ^{-1}\cup v_{S_0}.
%\end{align*}
%Then for each $X_i,$ we recognize the simple closed path $$X_i' :=(X_{i1}')(h_{X_i})(X_{i2}').$$ To make the notation work for all cases, we define $\gamma_0$ to be the empty set, to correct for the case when one or both ends of $h_{X_i}$ lie on $S_0.$

%As $X_{i1}'$ and $X_{i2}'$ are each contractible to a point, we can expand the boundary of the 2-handle for each $Y_j \in Y,$ by inserting $X_{i1}'$ and $X_{i2}'$ around each portion, $X_i,$ of the boundary of $Y_j.$ The cones collapse to points, the 1-handles collapse to edges and the 2-handles collapse to discs, giving precisely the presentation 2-complex, $K_P,$ for $P(D).$ By the Seifert -- van Kampen theorem we have  $$|P(D)|\cong\pi_1(\widetilde{M(D)})* F_k.$$

Now $\widehat{M}$ deformation retracts to a 2-complex where the 1-skelton is a wedge of 1-spheres (one for each dual 1-handle) and whose 2-cells are the cores of the 2-handles attached along the $Y$-curves expanded out to have boundary in the 1-skeleton. This is the ``cannonical 2-complex'' (or the ``presentation 2-complex") corresponding to the presentation $P(D).$ Thus $$|P(D)|\cong\pi_1(\widetilde{M})* F_k.$$
\end{proof}

Corollaries~\ref{SvK}-\ref{conecor} are all dependent on the set $X.$ Partition $X$ into $X_m$ and $X_s,$ where
$X=X_m\cup X_s,$ $X_m$ is a set of meridian curves for the compression body $V_X,$ and $X_s$ is the set of splitting curves, such that $S-X$ has $(\beta_0(X_s)+1)$ components.  Theorem~\ref{F_kcone} is the most general case: $X_m$ need not be a complete meridian set and $X_s$ need not be the empty set. Corollary~\ref{SvK} follows from Theorem~\ref{F_kcone} when $X_m$ is a complete meridian set and $X_s=\emptyset.$ This is also a result of the Seifert -- van Kampen theorem.  Corollary~\ref{F_k} follows from Theorem~\ref{F_kcone} when $X_m$ is a complete meridian set and $X_s\neq\emptyset.$ Finally, Corollary~\ref{conecor} follows from Theorem~\ref{F_kcone} when $X_m$ is not a complete meridian set but $X_s=\emptyset.$  We leave these corollaries as exercises for the reader.

\begin{cor}\label{SvK}
If $X$ is a complete set of meridian curves for $V_X,$  then $\pi_1(M(D))\cong |P(D)|.$
\end{cor}

\begin{cor}\label{F_k}
If $V_X$ is a handlebody and $\beta_0(X) - g(S)=k>0,$ then $|P(D)|\cong \pi_1(M(D)) * F_k.$
\end{cor}

\noindent \textit{Alternate Proof.}\quad Denote a set of meridian curves $X_m\subset X,$ and define $X_s:= (X-X_m)$ the set of non-meridian (or superfluous) curves, such that $\beta_0(X_m)=g$ and $\beta_0(X_s)=k.$  Then $S-X_m$ is a planar surface with one component, and each simple closed curve in $X_s$ separates $S-X_m.$ Therefore  $S-X$ is a collection of $k+1$ planar components, $S_0,S_1,\ldots,S_k$. Connect $S_0,\ldots,S_k$ with $k$ well-placed 1-handles, thereby creating $S^*,$ a $(g+k)-$genus splitting surface, and define $D^*:=(S^*;X,Y).$ Show first that $\pi_1(M(D^*))\cong |P(D^*)|\cong |P(D)|,$ and then that $\pi_1(M(D^*))\cong\pi_1(M(D)) * F_k.$

\begin{cor}\label{conecor}
Given a diagram $D,$ such that $S-X$ is connected and has positive genus, $|P(D)|\cong\pi_1(\widetilde{M}(D)).$
\end{cor}

Recall that a \textit{pseudo 3-manifold} is a triangulated 3-dimensional complex such that the link of every simplex is a connected manifold. Then we also have the following.

\begin{cor}
Every finitely presented group is the fundamental group of a pseudo 3-manifold.
\end{cor}

%%%%%%%%%%%%%%%%%%%%%%%%

Thus, for $P$ to naturally present a 3-manifold group, $S-X$ must be connected and $\partial M_X=\emptyset.$  From Theorem~\ref{F_kcone}, if $\beta_0(S-X)=k+1,$ then the fundamental group of the resulting manifold has a free factor of $F_k,$ but this corresponds to a connected sum:
$$\pi_1(M)*F_k\simeq \pi_1(M\# k(S^2\times S^1)),$$ which is still a 3-manifold group, provided $\partial_XM=\emptyset,$ and so we restrict our study to the case when $S-X$ is connected and turn our attention to when $\partial M_X=\emptyset.$ (See \cite{thesis} for an algorithm that determines connectedness.)  We also restrict our study to presentations that cannot be rewritten as the free product of two shorter presentations. (See \cite{thesis} for a simple test.)

In the next section we begin with $P$, and examine all diagrams associated with $P$ to determine if any of the diagrams have $\partial M_X=\emptyset.$

%%%%%%%%%%%%%%%%%%%%%%%%% THESIS PASTE %%%%%%%%%%%%%%%%%%%%%%%%%%%%%%%%%

\section{The class of diagrams  determined by a presentation, $[D(P)]$}\label{4}

In this section we focus on diagrams  determined by presentations which all reduce to the same $P.$   For $P$ to present the fundamental group of the 3-manifold determined by $(S;X,Y),$ we need only to see that the inner boundary on the $X$-side of the splitting is empty.  We remind the reader that we often depict diagrams on the split open surface $S-X$ projected into the plane. We are not guaranteed that this will be an embedding into the plane, but if $\partial_X M =\emptyset$ then $S-X$ can be drawn embedded in the plane.

\begin{defi}
A \textit{surface diagram} (or $s$-diagram) is a diagram $D$ such that $S-X$ is connected and $\partial_XM(D)=\emptyset$; so $P(D)$ naturally presents $\pi_1(M(D)).$
\end{defi}

\begin{defi}
A \textit{surface presentation} (or $s$-presentation) is a presentation that is determined by a surface diagram.
\end{defi}

Suppose $D$ is an $s$-diagram with a switchback. The relators recorded directly from $D$ would not be reduced because of the switchback. The presentation $P(D)$ would be reduced (by definition) and we would call $P(D)$ an $s$-presentation, even though if we began with $P$ we would not be certain we could recover $D$ in any bounded amount of time.

\begin{defi}
Let $D$ be a diagram with presentation $P(D)$ such that $$deg_A(P(D))=deg_G(D).$$ Then we call $D$ an \textit{exact s-diagram}, and $P$ an \textit{exact s-presentation}.
\end{defi}

From an exact $s$-presentation we can create an exact $s$-diagram (i.e. a diagram for a 3-manifold $M$ with $\partial_XM=\emptyset$), such that the diagram does not contain switchbacks and $S-X$ can be embedded in the plane. This also means that from an exact $s$-diagram, we can record a presentation and there will be no trivial reductions that occur in the relators. Both an $s$-presentation and an exact $s$-presentation are presentations that present groups that are isomorphic to the fundamental group of some 3-manifold. The difference between these two concepts is in how easy it is to recognize that property: for an exact $s$-presentation, we have an algorithm for constructing the exact $s$-diagram, but for the $s$-presentation, we are not guaranteed that we can determine the diagram within a specified time since the degree of the diagram may be arbitrarily large.

We want to know if a presentation is an $s$-presentation, however, we must be acutely aware that we only consider \textit{reduced} presentations, and that it may be a \textit{non-reduced} form of the presentation that reveals itself as an $s$-presentation.  Examples~\ref{EXAMPLE1} and \ref{EXAMPLE2} illustrate the problem: a presentation that does not appear to present a 3-manifold group, but an unreduced form of the same presentation that does present a 3-manifold group. In \S4.1, we take a topological look at the problem.  We introduce Whitehead homeomorphisms, Whitehead automorphisms and a theorem of Whitehead's in \S\ref{typeII3}, which we will use to partially answer our question, including revisiting the presentation from \S\ref{typeII1} and demonstrating that since we knew $P$ was an $s$-presentation, we can find a $P'$ differing from $P$ by Whitehead automorphisms such that $P'$ is an exact $s$-presentation. However, in \S\ref{typeII4} we give an example demonstrating that Whitehead's theorem is not enough to determine whether $P$ is an $s$-presentation in a finite amount of time.

\begin{remark}\textbf{ A note on the class of diagrams of fixed degree, $[D(P)]_d\subset [D(P)].$}
Creating a diagram from a presentation is not a well-defined process because there is a choice in the order that the $Y$-curves cross as you flow around an $X$-curve. This results in a finite class of diagrams of fixed degree determined by $P,$ denoted $[D(P)]_d.$ The degree $d$ diagrams are an important subset of $[D(P)],$ with a one-to-one correspondence between $[D(P)]_d$ and the class of permutation data sets of degree $d.$ The set of diagrams $[D(P)]_d$ contains all re-orderings of elements around each $X_i.$  We have developed an algorithm \cite{thesis} (and a computer program, available upon request) which can process any finite presentation $P$ and determine if any diagram in $[D(P)]_d$ is an exact $s$-diagram, meaning it can determine whether $P$ naturally presents a 3-manifold group.

If each of the $X_i$ curves has $k_i$ intersections with $Y$ curves, then there are $(k_i - 1)!$ ways to reorder $X_i$, and hence $\Pi _{i=1}^m(r_i-1)!$ elements in $[D(P)]_d$ to be run through the algorithm. This number can become extremely large, so while it is nice that we have a finite algorithm for testing the equivalence class of diagrams  determined by  $P,$ we see that these computations quickly become unwieldy.  The reader should note that this algorithm was implemented with a brute force exhaustion method. The use of such a method does not mean that there doesn't exist a clever method that would reduce the runtime. The complexity of this problem is still an open question.

We refer to the \textit{algorithm} in \cite{thesis} if we must determine whether a finite number of presentations or diagrams naturally present a 3-manifold group.  We turn our attention to $[D(P)]$ and refer the reader to \cite{thesis} for all treatment of $[D(P)]_d.$
\end{remark}

\subsection{An algebraic example illustrating why we must consider diagrams for unreduced presentations}\label{typeII1}
%%%%%%%%%%%%%%%%%%%%%

%
%%%%%%%%%%%%%%%%%%%%%%%%%

\begin{exam}\label{EXAMPLE1}
Consider the presentation $$P:=\langle x_1,x_2,x_3 : x_1 x_2^{-1} x_1^{-1}x_3x_2^{-1}x_1^{-1} x_3 x_1x_2^{-1} x_1^{-1} x_3 x_2 x_1 x_1 x_3^{-1} x_1 x_2^{-1} x_2^{-1}  \rangle,$$
with $deg_A(P)=18.$  Figure~\ref{example1pic} is the graph of a diagram for $P,$ and we conclude that $P$ is not (yet) recognized as an $s$-presentation as it is not planar. We assure the reader that no other diagram in $[D(P)]_d$ is planar (checked via the algorithm in \cite{thesis}).

\begin{figure}[ht!]
   \begin{center}
            \includegraphics[scale=.85]{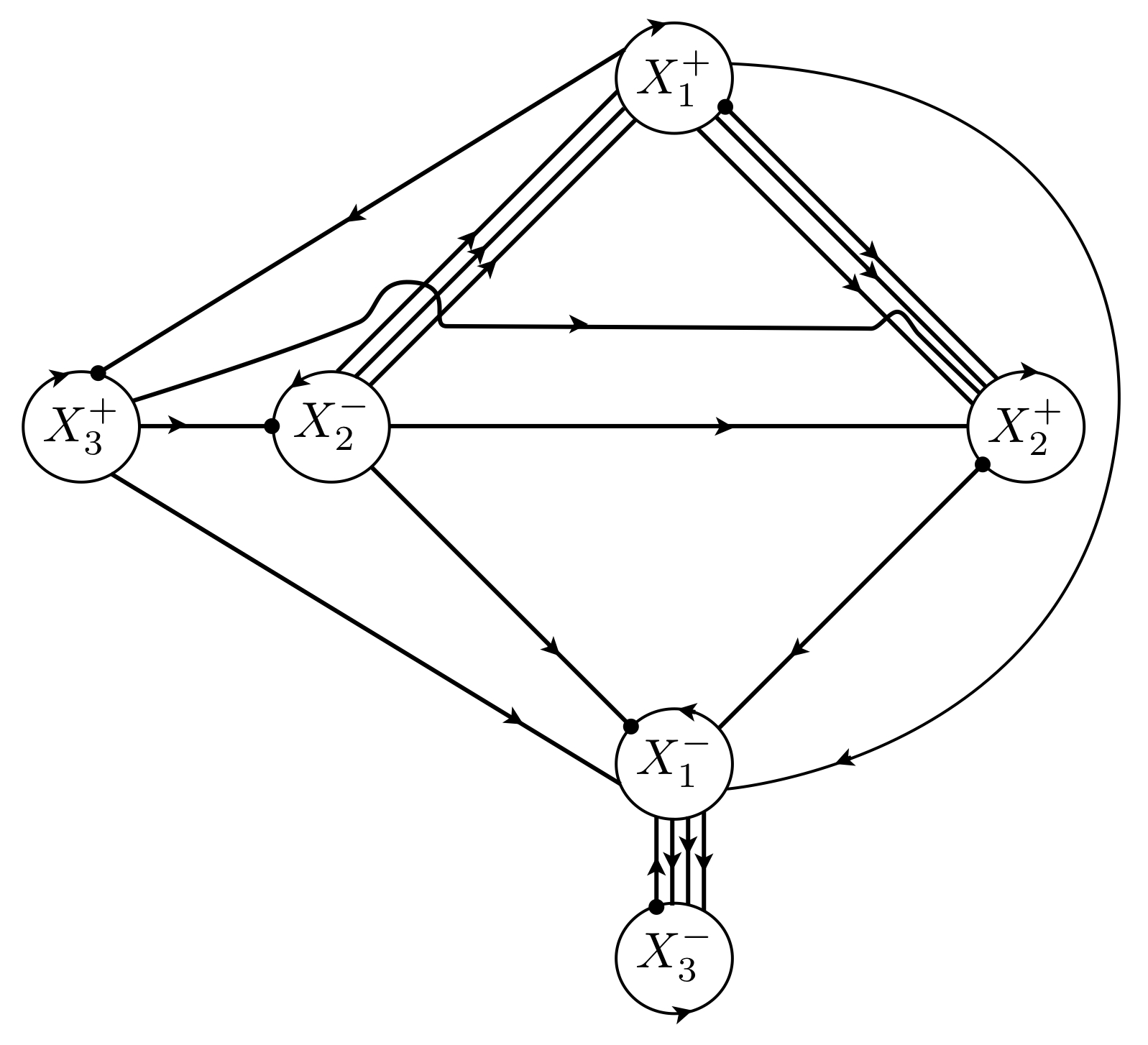}
        \caption{The diagram $D(P)$}\label{example1pic}
    \end{center}
\end{figure}
\end{exam}

\begin{exam}\label{EXAMPLE2}
Consider the presentation
$$P':=\langle x_1,x_2,x_3 : x_1 x_2^{-1} x_1^{-1}x_3 \underline{x_1^{-1}x_1} x_2^{-1}x_1^{-1} x_3 x_1x_2^{-1} x_1^{-1} x_3 x_2 x_1 x_1 x_3^{-1} x_1 x_2^{-1} x_2^{-1}  \rangle,$$
as shown in Figure~\ref{SWITCHBACKPIC}, with $deg_G(P')=20.$
Notice that $P'$ $P$ from Example~\ref{EXAMPLE1} differ only by a copy of $x_1^{-1}x_1$ (underlined). As our convention is to consider only reduced presentations, $P'$ reduces to $P.$

\begin{figure}[ht!]
   \begin{center}
            \includegraphics[scale=.85]{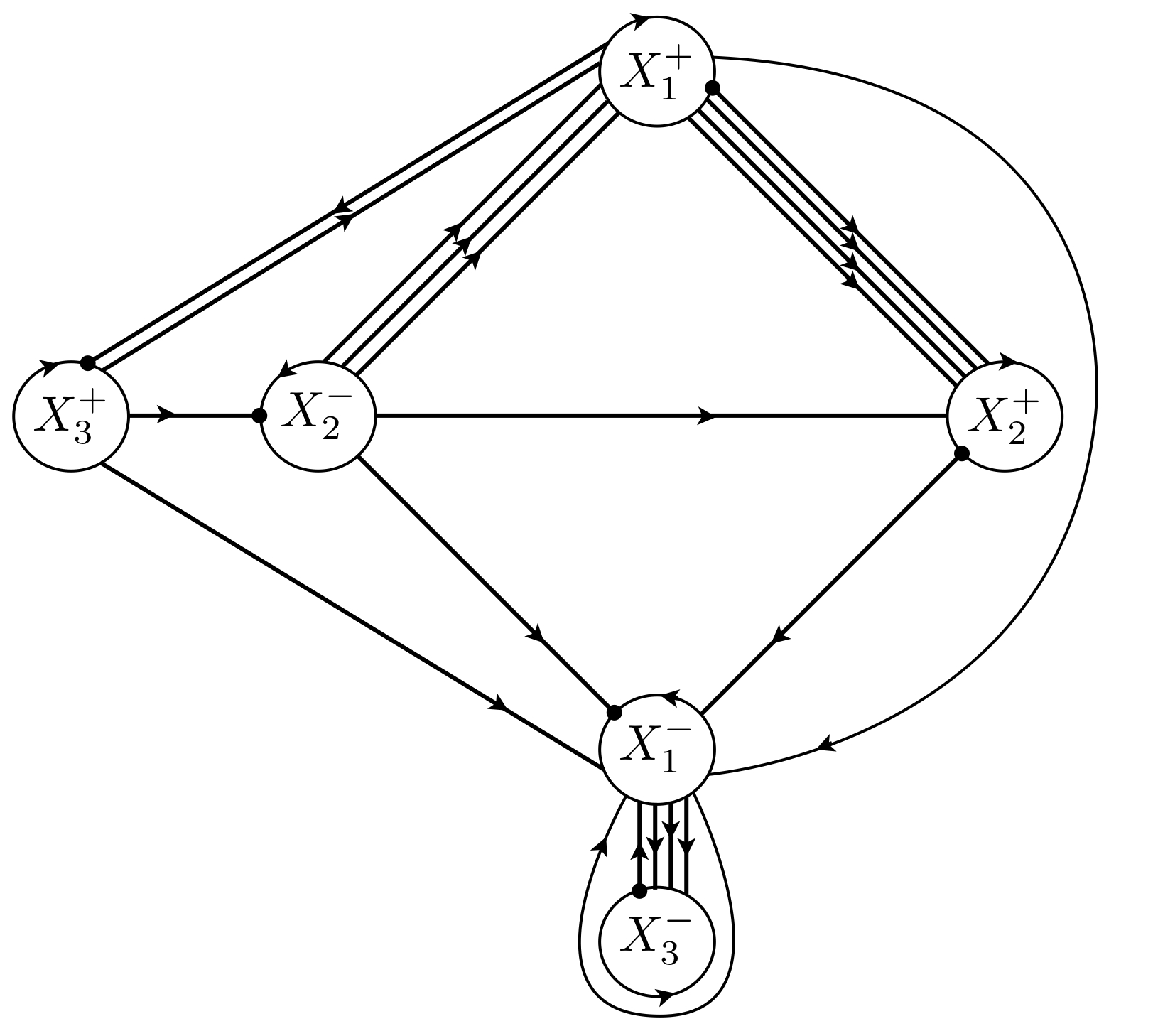}
        \caption{The diagram $D(\widetilde{P})$}\label{SWITCHBACKPIC}
    \end{center}
\end{figure}
\end{exam}

We determined $P$ was not an exact $s$-presentation (using the algorithm to check $[D(P)]_{18}$), but we did establish that $P'$ is an $s$-presentation, though not an exact $s$-presentation, since $deg_A(P')=18 \neq deg_G(P')=20.$ As $P$ and $P'$ yield the same reduced presentation, we can conclude that $P$ is an $s$-presentation.

Our methods thus far are not strong enough to conclude that a presentation is or is not an $s$-presentation, since to proceed in the same manner as the examples just given is possible, but time consuming. There are infinitely many forms of a non-reduced presentation to check, but by systematically testing a presentation (as per \cite{thesis}) after each insertion of combinations of $x_ix_i^{-1},$ we will find the planar, unreduced presentation if it exists.

%We conclude that our methods \textit{thus far} can only answer either (1) Yes, $P$ is an exact $s$-presentation, or (2) Inconclusive as to whether $P$ is an $s$-presentation.

\begin{lem}\label{r.e.}
The problem of deciding whether an arbitrary presentation naturally presents the fundamental group of a 3-manifold is recursively enumerable.
\end{lem}

Let us take a topological look at the problem of switchbacks. For any given diagram, no simplifications are performed other than pulling the two sets of curves ``tight" to remove bigons. Such an action keeps the curves in the same homotopy class. A diagram may contain switchbacks, coinciding with curves that loop the base of a handle before doubling back to cross a meridian, and as such cannot be homotoped away. When looking at $S-X$, the switchback appears as a curve based on a vertex. (For an example, note the switchback based on $X_1$ in Figure~\ref{SWITCHBACKPIC}.)  Thus a switchback is a portion of a $Y$-curve with both endpoints on the same side of an $X$-curve. Algebraically, this corresponds to a relator $y=\gamma_1 xx^{-1}\gamma_2,$ which trivially reduces to $y=\gamma_1\gamma_2,$ and therefore if we had started with the algebraic data we would not have known about the switchback. The trivial reduction reduces the algebraic degree so it is no longer equal to the geometric degree. An attempt to create a new diagram that corresponds to the reduced presentation may be fruitful --- in that it could create an exact $s$-diagram --- but it may not. The trivial reduction in the relator results in a new edge between the elements that came immediately before and after $xx^{-1}.$ This new edge may result in a diagram that cannot be embedded in the plane, meaning the diagram depicts a splitting with $\partial_X M\neq\emptyset.$

A trivial reduction could lead to a finite sequence of trivial reductions, if the elements immediately before and after $xx^{-1}$ are also inverses. What this demonstrates topologically is that, even if the curve that hosts a switchback is removed (by a certain type of handleslide move for instance), that removal may introduce another switchback, so our problem is not easily solved by topological methods. We now turn to algebraic methods.

\subsection{Whitehead homeomorphisms and automorphisms}\label{typeII3}
In this section we introduce Whitehead homeomorphisms and Whitehead automorphisms, following the notation as laid out in \cite{LS}. A Whitehead homeomorphism on a diagram corresponds to a Whitehead automorphism on the group presentation. We first discuss how the Whitehead homeomorphism and automorphism simultaneously influence one another.  We then introduce a theorem by Whitehead that states if the total relator length in a presentation can be reduced, then it can be monotonically reduced using Whitehead automorphisms.  Finally, we explicitly state how this helps us get a partial answer to our question.

Work by J. Nielson in 1924 \cite{JN24} demonstrated that the automorphism group $Aut(F_k)$ of the free group $F_k$ is generated by elementary transformations (now called \textit{elementary Nielson transformations}) and $T$-transformations (now called \textit{Whitehead automorphisms}).  Whitehead's work of 1936 (\cite{wh1},\cite{wh2}) followed on Nielson's work, but used  the theory of handlebodies to prove that one can effectively decide when two sets of cyclic words in $F_k$ are equivalent by $Aut(F_k).$  (See \cite{LS} for a modern survey of Nielson and Whitehead's work.) Two simplified proofs of Whitehead's result (\cite{ER},\cite{HL}) use purely algebraic techniques, and in fact we present below a reformulation of Whitehead's theorem, as given in \cite{ER}.

Let $V$ be an oriented handlebody, and $X=\{X_1, X_2, \ldots , X_g\}\subset \partial V$ a set of meridians for $V$ which determine a dual basis $\{x_1,x_2,\ldots , x_g \}$ for $\pi_1(X)$ in the standard way, such that $X_i$ and $x_i$ have a positive crossing as in Figure 3.

\begin{defi}
Suppose $Z$ is a simple closed curve in $\partial V - X$ and that for some $j,$ $X-\{X_j\}\cup\{Z\}$ is also a set of meridians for $V.$ Then there is a homeomorphism $h:V\rightarrow V$ such that $h(X_j)=Z$ and $h(X_i)=X_i$ for all $i\neq j.$ This is unique up to Dehn twists along the $\{X_i\}.$ If we require $h$ take a specified dual basis of $X-\{X_j\}\cup \{Z\}$ to a specified dual basis of $X,$ it is unique up to isotopy. We call $h$ a \textit{Whitehead homeomorphism}.
\end{defi}

If some arc of $Y$ consecutively crosses $Z$ in opposite directions, then the Whitehead homeomorphism will induce a switchback.

Let $Q:=\{\partial V \textrm{ split open along } X\}.$ So $\partial Q$ has $2g$ components, a pair $X_i^+, X_i^-$ for each $1\leq i\leq g$ where the orientation on $X_i^+$ induced by the orientation of $Q$ agrees with the orientation given by $X_i,$ and where the orientation on $X_i^-$ induced by the orientation of $Q$ is opposite to the orientation given by $X_i.$

The condition that $X-\{X_j\}\cup\{Z\}$ be a set of meridians for $V$ is that $X_j^+$ and $X_j^-$ lie in different components of $Q-\{Z\}.$ Let $H$ be the choice of the component of $Q-\{Z\}$ whose oriented boundary is $-Z.$
Let $A$ be the collection of the $X_i^{\pm}$ which lie in $H,$ and $A_0\in A$ be the one element of $\{X_j^+,X_j^-\}$ which lies in $H.$

Up to isotopy, we can assume $Z$ is a boundary component of a regular neighborhood of $A\cup T,$ where $T$ is a tree joining the components of $A.$  We assume our dual basis to $X-\{X_j\}\cup\{Z\}$ is $\{x_1',x_2',\ldots, x_{j-1}', z, x_{j+1}',\ldots , x_g'\}.$ Then $h_*(x_i)=x_i'$ for $i\neq j,$ and $h_*(x_j)=z,$ for $h_*$ the dual of $h.$

Obviously the dual basis has been changed to reflect the change in the meridian set from $X_j\mapsto Z,$ but we also are using $x'_i$ instead of $x_i,$ for $i \neq j.$ If $z$ did not cross an $x_i,$ then $x_i'=x_i,$ but if $x_i\cap z\neq \emptyset,$ then we would need a new curve $x_i'$ that is still dual to $X_i,$ but we are not guaranteed that $x_i$ and $x_i'$ are equivalent up to homotopy.

%ORIGINAL STATEMENT Moreover, $x_i\cap Z$ consists of 0, 1, or 2 points (near the point $X_i\cap x_i$) according as none, one or both $X_i^+, X_i^-$ lie in $A,$ together with pairs (near points of $x_i\cap T$) which lead to cancellations $zz^{-1},$ when writing $x_i$ in the basis $\{x_1', \ldots, x_{j-1}',z,x_{j+1}',\ldots\}$ so in $\pi_1(V)$ we have:

Moreover, $x_i\cap Z$ consists of 0, 1, or 2 points near the point $X_i\cap x_i$ (as neither, one or both $X_i^+, X_i^-$ lie in $A),$ together with pairs near points of $x_i\cap T,$ which lead to cancellations $zz^{-1}$ when writing $x_i$ in the basis $\{x_1', \ldots, x_{j-1}',z,x_{j+1}',\ldots\}.$ So in $\pi_1(V)$ we have:
\begin{equation*}
x_j= 
%\begin{split}
\begin{cases}
 z;  &\quad A_0=X_j^+\\
 z^{-1}; &\quad A_0=X_j^-,
\end{cases}
%\end{split}
\end{equation*}
and rewriting each $x_i$ in the new basis, we have

\begin{equation*}
x_i= 
%\begin{split}
\begin{cases}
z^{-1}x_i'z; &\quad X_i^+, X_i^- \in A \\
x_i'z;  &\quad X_i^+\in A, X_i^-\not \in A \\
z^{-1}x_i'; &\quad X_i^-\in A, X_i^+\not\in A \\
x_i'; &\quad X_i^{\pm}\not\in A.\\
\end{cases}
%\end{split}
\end{equation*}

We use the above to rewrite each relation, given by the $Y_j$'s,  in terms of the new generators. Note the automorphism $h_*: \pi_1(V)\rightarrow\pi_1(V)$ depends only on $A$ and $A_0,$  whereas the particular Whitehead homeomorphism inducing it depends on the choice of the isotopy class of $Z$. Two Whitehead homeomorphisms thus result in the same Whitehead automorphism if they have the same $A$ and $A_0.$ We denote this automorphism $(A,A_0)$

Suppose also $Y=\{Y_1,Y_2,\ldots\}$ is a set of oriented, disjoint, simple closed curves in $\partial V$ determining a set of conjugacy classes in $\pi_1(V).$ We assume $Y$ meets $X$ efficiently (no bigons). To see the geometric effect on $Y$ of a particular Whitehead homeomorphism inducing a given automorphism $(A,A_0),$ we chose $Z$ judiciously: consider the graph in $S^2$ whose vertices are the components of $S^2-Q$ (identified with the boundary components) and whose edges are the $Y$-stacks.  Let $T$ be a tree in $\Gamma$ containing exactly the vertices in $A$ and let $Z$ equal the boundary of a regular neighborhood of $T$ and the vertices. Then

\begin{enumerate}
\item $Y$ meets $X'=(X-X_j)\cup Z$ efficiently so $h(Y)$ meets $X$ efficiently. So $$deg_G(h(Y),X)=deg_G[(Y,X)+\beta_0(Y\cap Z)-\beta_0(Y\cap X_j)].$$
\item If $Y$ has no switchbacks ($Y$-stacks with both ends in the same $X_i^{\pm}$) and $Z$ crosses each $Y$-stack at most once, then $h(Y)$ has no switchbacks.
\end{enumerate}

A restatement of Whitehead's theorem by Rapaport is as follows. (Only the portion of the theorem relevant to this work has been included. See \cite{ER} and references therein.)
\begin{thm}\label{WhThm}
Given a set of words $W_0,\ldots,W_k$ in the generators of $F_n,$ if the sum $L$ of the lengths of these words can be diminished by applying automorphisms of $F_n$ to the generators, then it can also be diminished by applying an automorphism of a preassigned finite set of automorphisms (the so-called $T$-transformations).
\end{thm}

To make this consistent with our notation, we note that the $T$-transformations of \cite{ER} are precisely the Whitehead automorphisms defined in this paper. Thus, given a finite presentation $P,$ if $deg_A(P)$ can be reduced, then it can be reduced by a Whitehead automorphism. This result provides a partial answer to our question: is a given presentation naturally an $s$-presentation? Meaning, does there exist a presentation $P'$ differing from $P$ only by Whitehead automorphisms, such that $P'$ is an exact $s$-presentation?  If $deg_A(P')<deg_A(P),$ then according to Whitehead's result, we can find $P'$ by monotonically reducing the algebraic degree of $P.$ We next spell out how to search for such a $P'.$ This will be a finite check, since after each degree-reducing Whitehead automorphism we can process the new presentation using the methods of \S\ref{4}.

We next develop the notation to recognize a Whitehead automorphism that will reduce the algebraic degree. Let $n(x_i)$ denote the number of occurrences of $x_i$ and $x_i^{-1}$ in the relators of $P,$ and $n(x_{i_1}x_{i_2})$ denote the number of occurrences of $x_{i_1}x_{i_2}$ and $(x_{i_1}x_{i_2})^{-1}$ in the relators of $P.$

We now partition the generators based on whether each generator and its inverse is in the set $A$ (or not in the set $A$).  Let $B:=\{b_1, b_2,\ldots\}$ be the finite list of $X_i$ such that $X_i^+\in A,$ but $X_i^-\not\in A.$ Let $C:=\{c_1, c_2,\ldots\}$ be the finite list of $X_i^-$ such that $X_i^-\in A,$ but $X_i^+\not\in A.$ Let $G:=\{g_1, g_2,\ldots\}$ be the finite list of $X_i$ such that both $X_i^+,X_i^-\in A.$

Such a Whitehead automorphism $h_*$ as spelled out above has the following effect on the algebraic degree of the presentation $P:$
\begin{itemize}
\item For each occurrence of $x_j$ or $x_j^{-1}$ in a relator, the algebraic degree of $P$ will \textbf{not change}, as we are replacing a single variable with a single variable.
\item For each $b\in B,$ the algebraic degree of $P$ will \textbf{increase} by $n(b),$ as each $b$ or $b^{-1}$ forces the insertion of an extra variable.
\item For each $c\in C,$ the algebraic degree of $P$ will \textbf{increase} by $n(c),$ as each $c$ or $c^{-1}$ forces the insertion of an extra variable.
\item For each $g\in G,$ the algebraic degree of $P$ will \textbf{increase} by $2[n(g)],$ as each $g$ or $g^{-1}$ forces the insertion of two extra variables.
\item The algebraic degree of $P$ will be \textbf{reduced} by 2, for each occurrence of cyclically adjacent elements in a relator:
	\begin{itemize}
	\item $x_jb^{-1}$ or $(x_jb^{-1})^{-1}$ for each $b\in B$
	\item $x_jc$ or $(x_jc)^{-1}$ for each $c\in C$
	\item $x_jg$ or $(x_jg)^{-1}$ for each $g\in G.$
	\end{itemize}
\end{itemize}

So by merely knowing that the total number of reductions in the relators of $P$ is greater than the total length increase, our Whitehead automorphism will result in a presentation with algebraic degree less than that of $P.$

\begin{lem}\label{lyndon}
A Whitehead automorphism, with isomorphism and sets $B, C$ and $G$ as defined above, will reduce $deg_A(P)$ if and only if
$$\sum_{b\in B} n(b) + \sum_{c\in C} n(c) + \sum_{g\in G} 2n(g) <  \sum_{b\in B} 2n(x_jb^{-1}) + \sum_{c\in C} 2n(x_jc) + \sum_{g\in G} 2n(x_jg).$$
\end{lem}

The process by which we would determine if the inequality of Lemma \ref{lyndon} holds for a presentation is programmable, so we can generate the list of presentations with decreasing algebraic degree until the algebraic degree can no longer be shortened. This programmability is fantastic because if we had attempted this same feat at the topological (diagram) level, there would have been a lot of choice involved, as the choice of the bonding tree affects the degree of the diagram.

This next proof uses wave reduction. A \textit{wave} is a curve in $S-Y$ meeting $X$ only in its endpoints such that neighborhoods of the endpoints of the curve are both on the same side of an $X$-curve, say $X_i,$ and such that the curve cannot be homotoped back into $X.$ The endpoints of a wave break $X_i$ into two arcs. When a diagram has a wave, we can replace $X_i$ with the wave union one of the two arcs, giving a new set of $X$-curves. If a diagram has a switchback, then the diagram has a wave. Notice that a wave reduction is a Whitehead homeomorphism where $A_0$ is the curve that supports the switchback (or the wave), and $A$ is the set of curves in the component of $Q-Z$ containing $A_0.$

\begin{lem}\label{HaveExact}
If $D$ is an $s$-diagram, then there exists a $D'$ such that
\begin{enumerate}
\item $D'$ is an exact $s$-diagram;
\item $deg_G(D')\leq deg_G(D);$
\item $D$ and $D'$ differ by a sequence of Whitehead homeomorphisms.
\end{enumerate}
\end{lem}

\begin{proof}
Let $D$ be an $s$-diagram (but not exact). Then $D$ contains a switchback. A wave reduction move will remove a switchback and decrease the geometric degree by strategically replacing an arc that crosses a $Y$-stack with one that does not. We only need to show that a wave move corresponds to a sequence of Whitehead homeomorphisms.

The wave will have endpoints in some component of $X$, but is not parallel to $X$ by a homeomorphism. The $X$-curve will have two arcs, and we replace the arc (with boundary that of the wave) which maintains the linear independence of the homotopy classes of the $X$-curves.  The simple closed curve $Z$ will hit $Y,$ but the new curve hits $Y$ less than the previous curve. This sequence of wave reductions is precisely a single Whitehead homeomorphism, where $A$ is the set of all curves included in the side of each replaced arc, and $A_0$ is the final curve that was replaced in the sequence.
\end{proof}

\begin{cor}
If a presentation $P$ is an $s$-presentation, then there exists a presentation $P'$ such that
\begin{enumerate}
\item $deg_G(P'(D))<deg_G(P(D)),$
\item $P'$ is an exact $s$-presentation,
\item and $P$ and $P'$ differ by a sequence of Whitehead automorphisms.
\end{enumerate}
\end{cor}

Consider now the genus 2 case, $X=\{X_i, X_j\}.$ The Whitehead automorphisms just described are even simpler, as the bonding tree cannot connect both $X_j^+$ and $X_j^-,$ meaning $A$ can contain $X_j^+$ and one or both of $X_i^+, X_i^-.$  However, if $A$ contained both $X_i^+, X_i^-,$ then the complement would contain only $X_j^-$ and the replacement would do nothing. Therefore, in the genus 2 case, $A$ contains $X_j$ and $X_i^{\epsilon},$ making the isomorphism less tedious, and reducing the number of adjacent pairs to record in the relators. Since the Whitehead homeomorphism only depends on $A$ and $A_0$ we can (for the two-generator case) always realize it by a Whitehead homeomorphism that reduces geometric degree.

\begin{lem}\label{KPP}
A Whitehead automorphism, with isomorphism as defined above on a two-generator presentation, will reduce $deg_A(P)$ if and only if $n(x_i)< 2n(x_jx_i^{-1}).$
\end{lem}

\begin{exam}
We now refer back to Examples~\ref{EXAMPLE1} and \ref{EXAMPLE2}. The presentation $P$ given in that section (\S4.1) is an $s$-presentation since an unreduced form of $P$ is an $s$-presentation. By Lemma~\ref{HaveExact}, we know there is an exact $s$-presentation $P',$ which we can find using Lemma~\ref{HaveExact} on presentation $P',$  or by using Lemma~\ref{lyndon} on presentation $P$ or $P'.$

The proof of Lemma~\ref{HaveExact} dictates $(A,A_0):$ $A_0$ is the curve $X_j^{\epsilon}$ in the cut-open surface $S-X$ upon which the the switchback is based, and $A$ is the set of curves $X_i^{\epsilon}$ bounded by $A_0$ and the switchback. Thus from Figure~\ref{SWITCHBACKPIC} for $P',$ $A_0:=X_1^-$ and $A:=\{X_3^-\},$ determining the following automorphism:

\begin{align*}
%x_1^{-1}\mapsto z\\
x_1&\mapsto z^{-1},\\
x_3&\mapsto z^{-1}x_3,\\
%x_3^{-1}\mapsto x_3^{-1}z\\
x_2&\mapsto x_2.\\
%x_2^{-1}&\mapsto x_2^{-1}.
\end{align*}

We verify that the automorphism will result in reduction of the geometric degree. We have the set $B=\emptyset,$ $C=\{x_3\},$ and $G=\emptyset.$ Since $n(c)=4,$ and $n(x_1^{-1}x_3)=4,$ we satisfy Lemma~\ref{lyndon} and we know that the net result on will be to reduce the geometric degree of $P$ by four. Performing this automorphism on $P$ we get the following:
\begin{align*}
P'&=\langle z,x_2,x_3 : z^{-1} x_2^{-1} z(z^{-1}x_3)x_2^{-1}z (z^{-1}x_3) x_1x_2^{-1} z (z^{-1}x_3) x_2 z^{-1}z^{-1} (x_3^{-1}z) z^{-1} x_2^{-1} x_2^{-1}  \rangle\\
&=\langle z,x_2,x_3 : z^{-1} x_2^{-1} x_3x_2^{-1}x_3 x_1x_2^{-1}x_3 x_2 z^{-1}z^{-1} x_3^{-1} x_2^{-1} x_2^{-1}  \rangle.\\
\end{align*}

\noindent Note $deg_G(P)=18$ and $deg_G(P')=14=deg_G(P'),$ as $P'$ is an exact $s$-presentation.

\begin{figure}[ht!]
   \begin{center}
  %  \begin{picture}(600, 100)(20,20)
            \includegraphics[scale=.7]{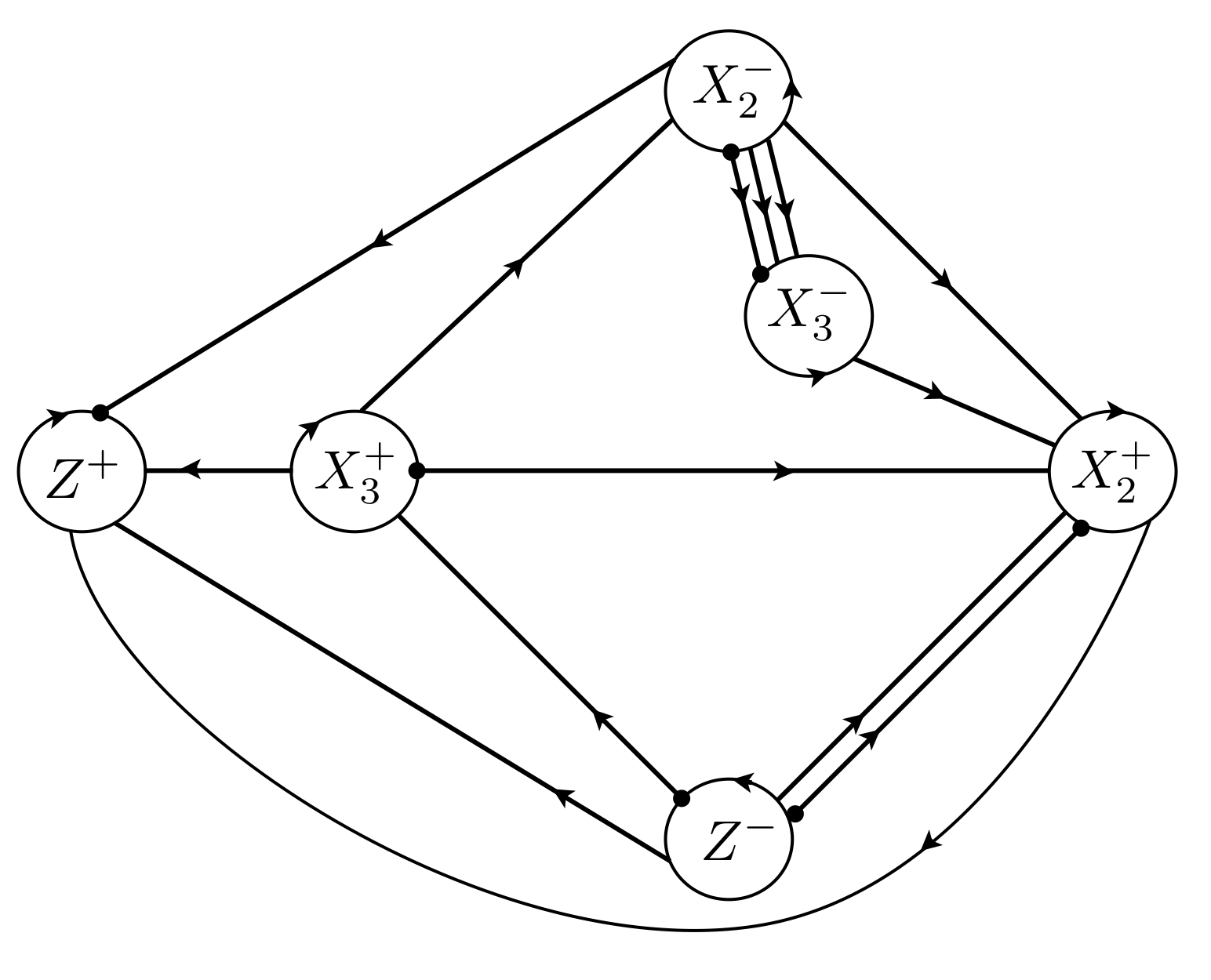}

      %  \end{picture}
        \caption{The diagram $D(P')$}\label{P'pic}
    \end{center}
\end{figure}

Suppose we were given $P,$ unaware that it was an $s$-presentation. Without knowing $P',$ we could have searched for all automorphisms that would reduce $deg_G(P)$ (using Lemma~\ref{lyndon}), applied each automorphism to $P,$ and checked whether it resulted in an exact $s$-diagram.
\end{exam}

\subsection{A demonstration that Whitehead's theorem is not enough to determine whether $P$ is an $s$-presentation in a finite amount of time}\label{typeII4}

We have just demonstrated that if there exists an exact $s$-presentation, $P',$ differing from $P$ by a sequence of Whitehead automorphisms with $deg_A(P')<deg_A(P),$ then we will find $P'$ in a finite amount of time.  However, suppose we have a sequence of Whitehead homeomorphisms taking $D$ to $D',$ such that each Whitehead homeomorphism decreases the geometric degree of the diagram, and the end result is a diagram with empty $X$-boundary and $deg_G(D')=deg_A(P(D')).$ Must the sequence have decreased the algebraic degree?  Unfortunately not, as the following sequence of examples demonstrates.

\begin{exam}\label{good}
Consider the positive Heegaard diagram, $D_1$ (Figure~\ref{D_1pic}).  Clearly $S-X$ is connected and embeds in $S^2,$ so this is an exact $s$-diagram with $deg_G(D_1)=15.$ The presentation  determined by $D_1$ is
$$P(D_1):=\langle x_1, \,x_2 \,:\, x_1^2x_2x_1x_2x_1^4,\, \,x_1x_2^3x_1x_2 \rangle,$$
so $deg_A(P(D_1))=15.$ It will always be the case that we have an exact $s$-presentation when we begin with a positive Heegaard diagram, as there is no possibility of a switchback. (See \cite{hempelpaper} for a discussion of positive Heegaard diagrams.)
\begin{figure}[ht!]
   \begin{center}
  %  \begin{picture}(600, 100)(20,20)
            \includegraphics[scale=.9]{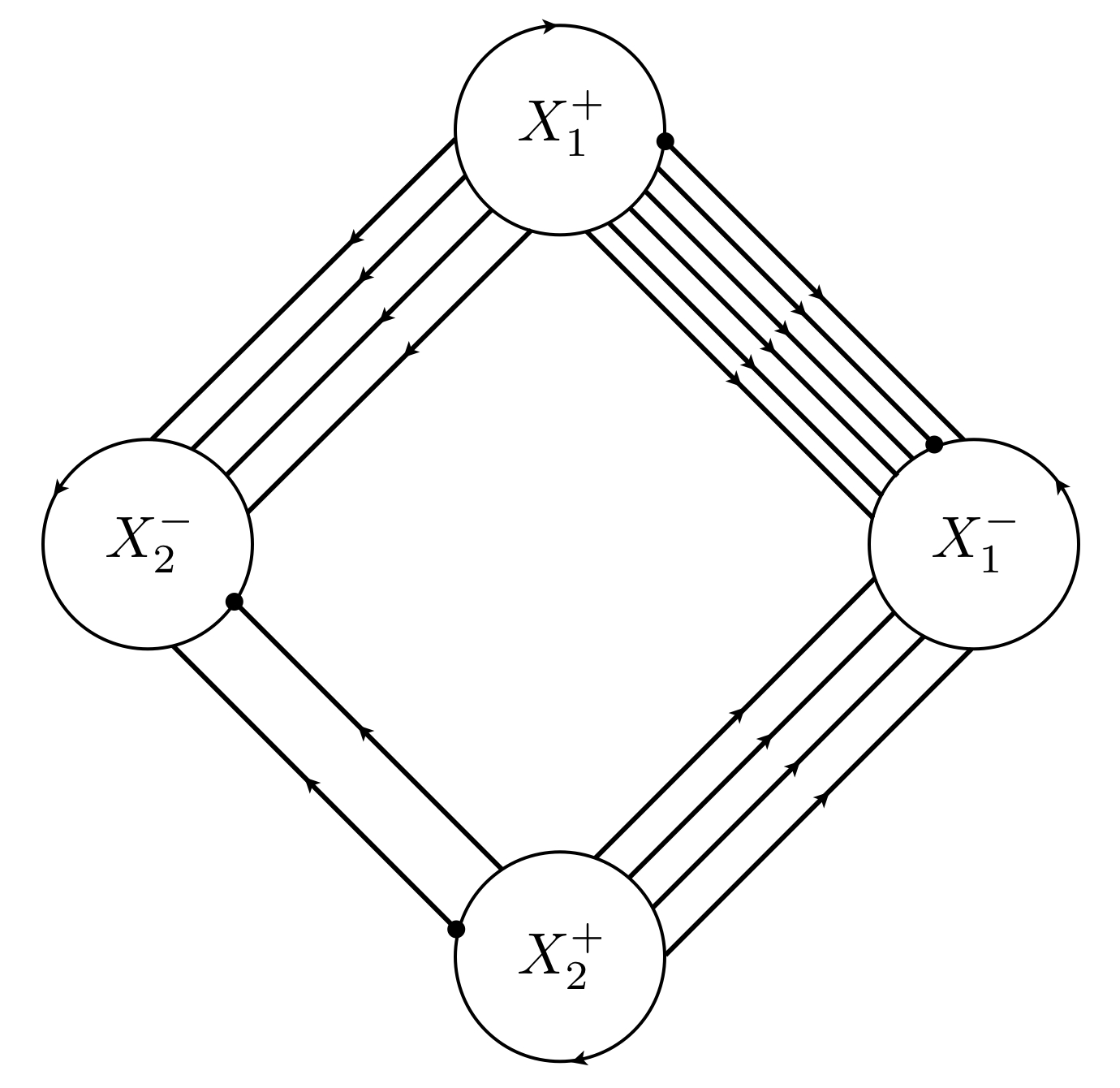}

      %  \end{picture}
	        \caption{The diagram $D_1$}\label{D_1pic}
    \end{center}
\end{figure}
\end{exam}

If we had started with or arrived at this diagram, we would be done. But what if we had started with a diagram that is a few Whitehead homeomorphisms away from this diagram? Would we still be able to determine that we had an $s$-presentation in some bounded amount of time?  In the next two examples, we make this diagram worse to illustrate the uncertainty of how to proceed because we are not guaranteed a monotonic reduction in algebraic and geometric degree. The reader will note that after each Whitehead homeomorphism we relabel the permutations beginning with 1, and rename $Z$ by $X_1$ to maintain a set of $X$-curves, $\{X_1, X_2\}.$

\begin{exam}\label{bad1}
By using the simple closed curve $Z$ (shown in Figure~\ref{D_1Zpic}) with $X_1\mapsto Z$, we get another $s$-diagram, $D_2$ (Figure~\ref{D_2pic}).

\begin{figure}[ht!]
   \begin{center}
            \includegraphics[scale=.9]{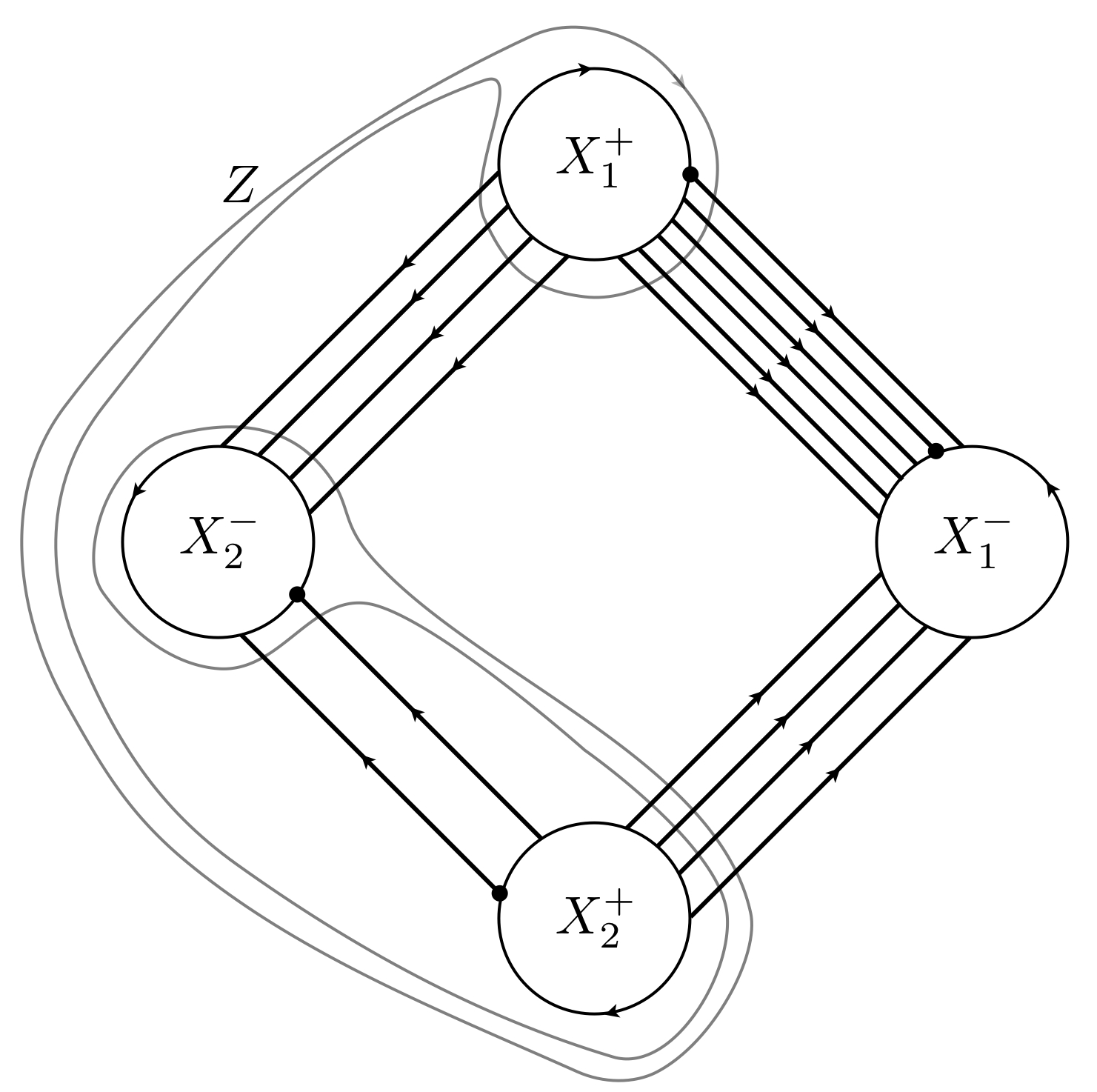}
        \caption{The diagram $D_1$ with curve $Z$}\label{D_1Zpic}
    \end{center}
\end{figure}

We see $deg_G(D_2)=29.$ The $s$-presentation  determined by $D_2$  is
\begin{align*}
\langle x_1, \,x_2\,: \,x_1 \underline{x_1 x_1^{-1}} x_2 \underline{x_1^{-1} x_1 }\underline{x_1 x_1^{-1}} x_2 \underline{x_1^{-1} x_1} x_1 x_1 x_1 x_1, \,\, \underline{x_1 x_1^{-1}} x_2 x_1^{-1} x_2 x_1^{-1} x_2 \underline{x_1^{-1} x_1} \underline{x_1 x_1^{-1}} x_2 \underline{x_1^{-1} x_1} \rangle,
\end{align*}
which reduces by performing trivial eliminations (underlined in $P(D_2)$), giving
$$P(D_2):=\langle x_1, \,x_2\,: \,x_1 x_2^2 x_1^4,    \,\, x_2 x_1^{-1} x_2 x_1^{-1} x_2^2\rangle.$$
\begin{figure}[ht!]
   \begin{center}
  %  \begin{picture}(600, 100)(20,20)
            \includegraphics[scale=.9]{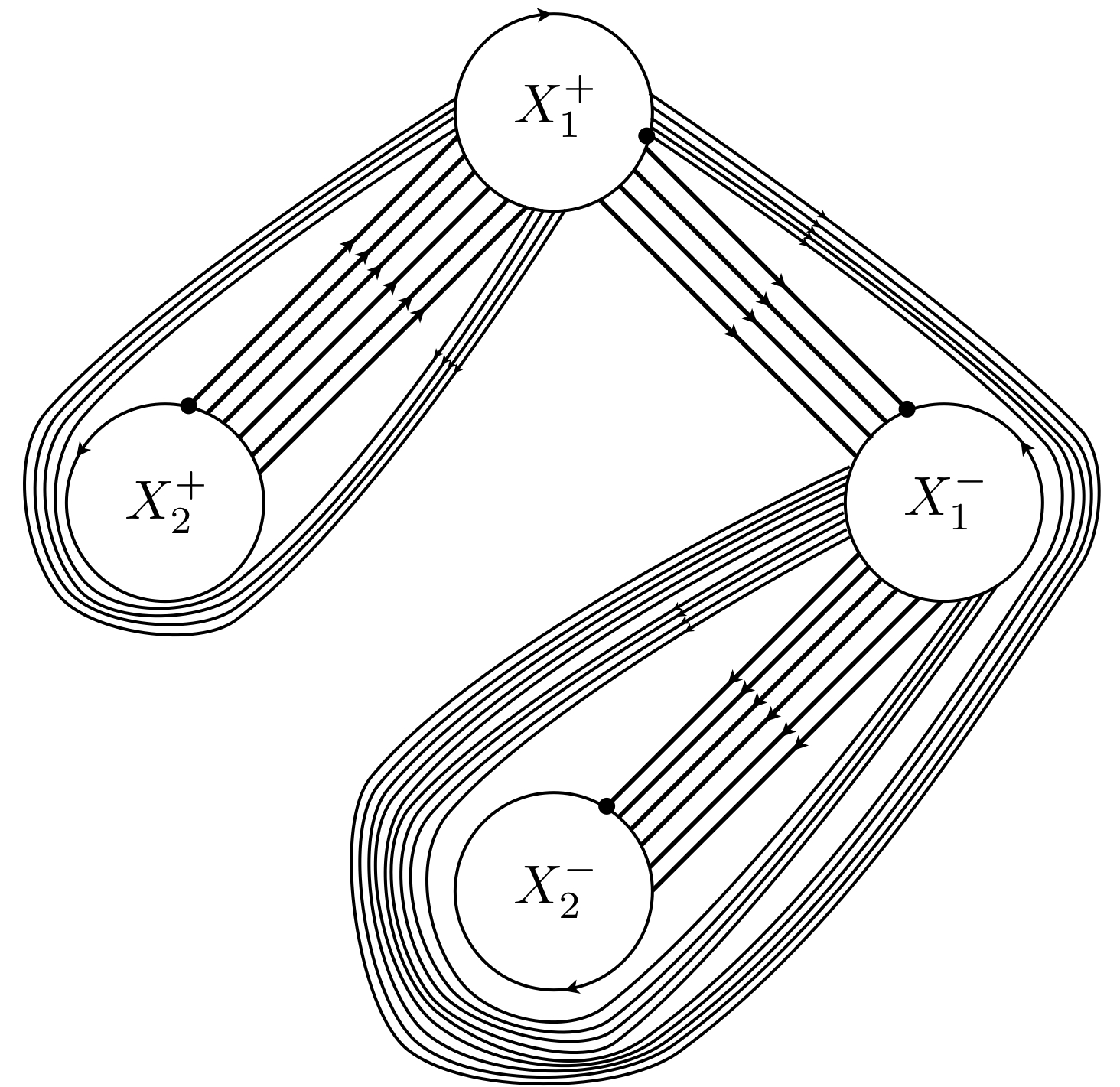}

      %  \end{picture}
        \caption{The diagram $D_2$}\label{D_2pic}
    \end{center}
\end{figure}
Note $deg_A(P(D_2))=13.$ As $deg_G(D_2)\neq deg_A(P(D_2)),$ $P(D_2)$ is not an exact $s$-presentation. However, since we were able to see the $s$-diagram with switchbacks for $P(D_2)$, we know it is an $s$-presentation.
\end{exam}

We perform another Whitehead homeomorphism, now on Example~\ref{bad1}, to further demonstrate how difficult it could be to recover an exact diagram that may only be two Whitehead homeomorphisms away from our current diagram.

\begin{exam}\label{bad2}
By using the simple closed curve $Z$ on $D_2$ (shown in Figure~\ref{D_2Zpic}) with $X_1\mapsto Z$, we get another $s$-diagram, $D_3$ (Figure~\ref{D_3pic}). We see $deg_G(D_3)=43.$

\begin{figure}[ht!]
   \begin{center}
            \includegraphics[scale=.9]{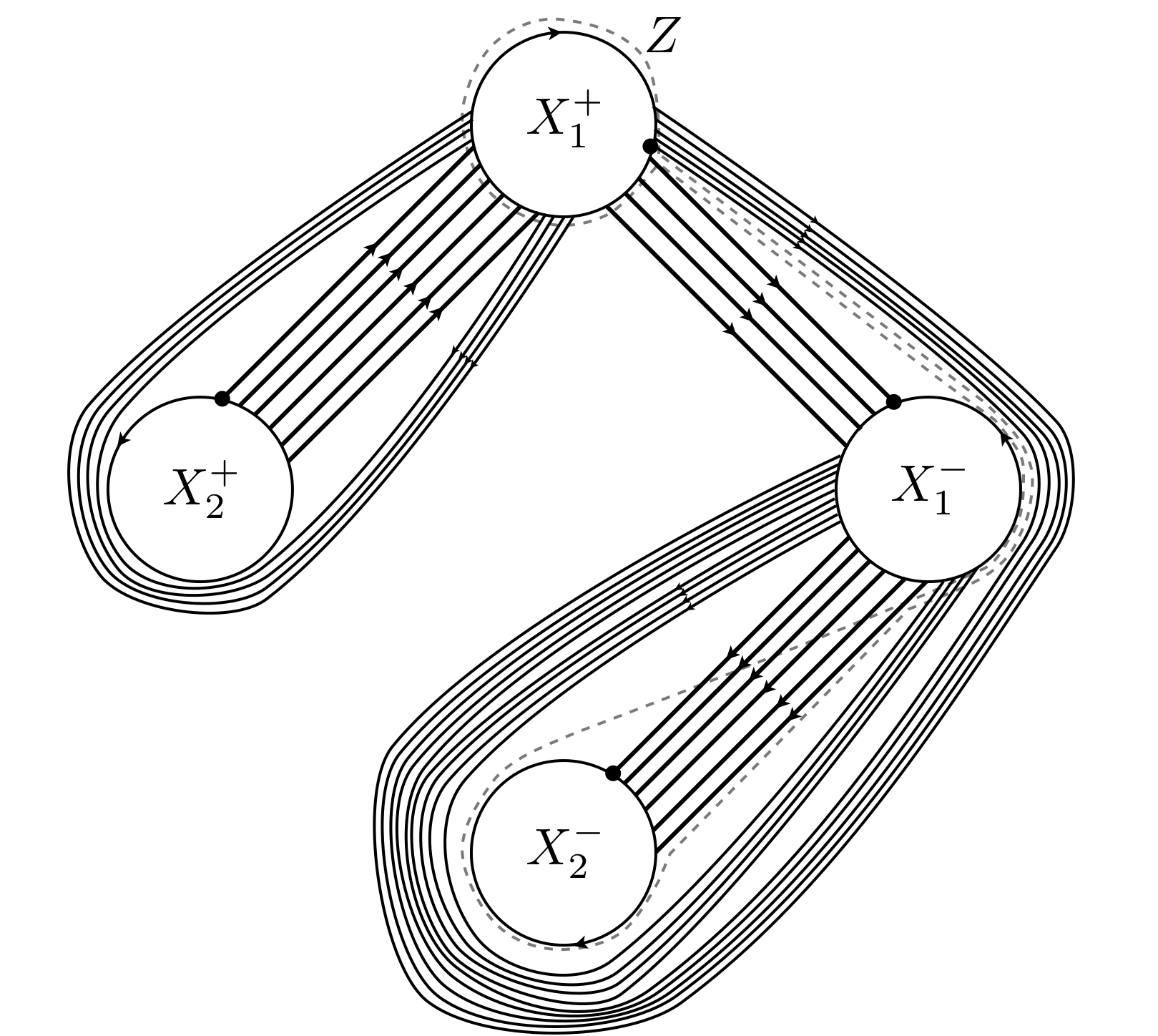}
        \caption{The diagram $D_2$ with curve $Z$}\label{D_2Zpic}
    \end{center}
\end{figure}

To determine the algebraic degree we first create the presentation directly from the diagram, $\langle x_1,\,x_2: r_1, \,r_2\rangle,$ where the relations $r_1$ and $r_2$ are

$$r_1:=  \underline{x_1 x_1 x_1^{-1} x_1^{-1}} x_2 \underline{x_1^{-1} x_1^{-1} x_1 x_1} \underline{x_1 x_1^{-1}} x_1^{-1} x_2 \underline{x_1^{-1} x_1^{-1} x_1x_1} x_1 x_1 x_1 x_1,$$
and
$$r_2:=    \underline{x_1x_1^{-1}}x_1^{-1}x_2x_1^{-1}\underline{x_1^{-1}x_1}x_1^{-1}x_1^{-1}x_2\underline{x_1^{-1}x_1^{-1}x_1x_1}\underline{x_1x_1^{-1}}x_1^{-1}x_2\underline{x_1^{-1}x_1^{-1}x_1x_1}.$$
We remove all trivial cancellations (shown underlined), getting the presentation

$$P(D_3)=\langle x_1, \,x_2\,:\,   x_2 x_1^{-1} x_2 x_1^4,  \,\,x_1^{-1}x_2x_1^{-3}x_2x_1^{-1}x_2\rangle.$$  Note that $deg_A(P(D_3))=15.$ Again, as $deg_G(D_3)\neq deg_A(P(D_3)),$ $P'(D_3)$ is not an exact $s$-presentation.

\begin{figure}[ht!]
   \begin{center}
            \includegraphics[scale=.9]{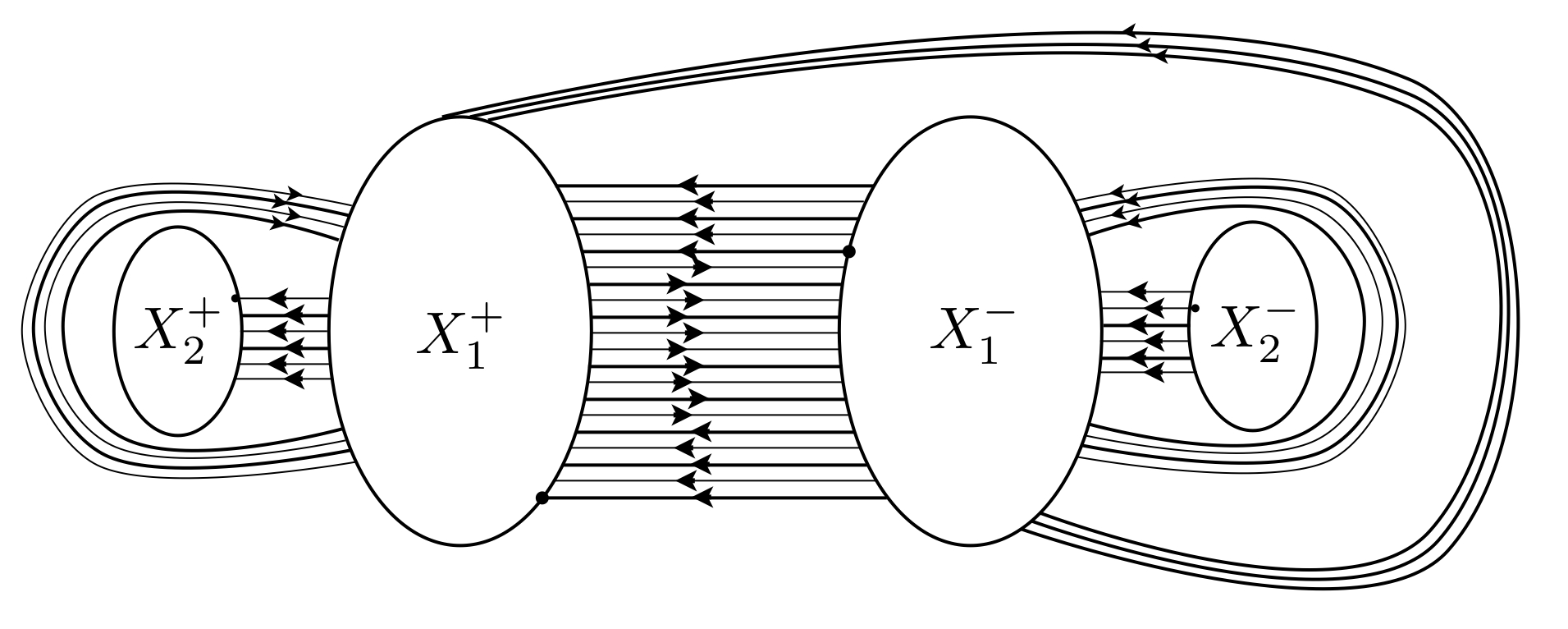}
        \caption{The diagram $D_3$}\label{D_3pic}
    \end{center}
\end{figure}
\end{exam}

Since we know $P(D_3)$ and $P(D_1)$ differ by Whitehead automorphisms, and that $P(D_1)$ is an exact $s$-presentation, then we know $P(D_3)$ is an $s$-presentation. However, consider if we had started with $D_3$ and proceeded to $D_1$ (see Table~\ref{D_itable}).

\begin{table}[htdp]
\begin{center}
\begin{tabular}{ccc}
Diagram & $deg_G(D_i)$ & $deg_A(P(D_i))$\\
\hline
$D_3$ & 43 & 15 \\

$D_2$ & 29 & 13 \\

$D_1$ & 15 & 15 \\
\end{tabular}
\caption{Algebraic and geometric degree for $D_i$}\label{D_itable}
\end{center}
\label{default}
\end{table}%

Hence we see there exists an example of a sequence of Whitehead homeomorphisms on a diagram that monotonically reduces the geometric degree and does not monotonically reduce the algebraic degree as the diagram is transformed to an exact $s$-diagram. There are two diagrams of the same algebraic degree, one of which gives an $s$-presentation and one of which gives an exact $s$-presentation, and one is not derived from the other simply by cancelling switchbacks.

This sequence of examples demonstrates that we cannot guarantee that Whitehead homeomorphisms that monotonically reduce the algebraic degree will take us to an exact $s$-diagram. This is a problem because there can exist an exact $s$-presentation $P'$ of greater algebraic degree than $P,$ and we have no way of systematically searching for $P'.$ More importantly, we have no way of knowing when we should stop searching for such presentations, and concluding that $P$ is not an $s$-presentation.

\begin{question}
If $D$ is a diagram of minimal geometric degree up to equivalence of Whitehead homeomorphisms, does $D$ have a minimal algebraic degree among all of the diagrams in this same class?
\end{question}

\begin{question}
If $D$ is an $s$-diagram, is there a bound on the geometric degree for an exact $s$-diagram of $D',$ such that $D'$ differs from $D$ by a sequence of Whitehead homeomorphisms.
\end{question}

\subsection{The case of the two-generator presentation is completely solved}\label{perm6}

We now consider the case when $P=\langle x_1, x_2 : r_1,\ldots,r_n\rangle.$ Any diagram  determined by $P$ will have two base curves, $X_1, X_2,$ and four vertices in the  geometric graph of the diagram. We have a complete solution to our problem in the two-generator case due to the unique topology of the geometric graph when it has only four vertices.

We continue the notation from \S\ref{typeII3}: $Z$ is the simple closed curve of the Whitehead homeomorphism and $A$ is the set of vertices within $Z.$ Without loss of generality, suppose out of the meridian set $X:=\{X_1,X_2\}$ that $A_0=X_1.$ Then $X_1^+$ and $X_1^-$ must be on different sides of $Z,$ and using either side of $Z$ for replacement will still contain the element to be replaced.  Also, we cannot have both $X_2^+$ and $X_2^-$ on the same side of $Z.$ Otherwise there would be a single $X_1^{\epsilon}$ on one side of $Z,$ forcing the Whitehead homeomorphism to replace that single curve, which would just be a renaming of elements.

When $\beta_0(X)\geq 3,$ $X_1^+$ and $X_1^-$ will still be separated by $Z,$ but we are not guaranteed that the remaining $X_j$ curves will be symmetric on either side of $Z.$  The case when $\beta_0(X)=2$ is special though, because there are only two ways in which $Z$ can separate the signed meridians, up to a relabeling of elements.
\begin{enumerate}
\item $Z^+\textrm{ surrounds } \{X_1^+, X_2^+\} \textrm{ and } Z^-\textrm{ surrounds } \{X_1^-, X_2^-\}. $
\item $Z^+\textrm{ surrounds } \{X_1^-, X_2^+\} \textrm{ and } Z^-\textrm{ surrounds } \{X_1^+, X_2^-\}. $
\end{enumerate}
Clearly renaming $W:=Z^-$ can change any issues with orientation, as well. When performing a Whitehead homeomorphism, these are the only choices we have for $A,$ so it is easy to try all of these Whitehead homeomorphisms and determine which homeomorphism gives a diagram of lower geometric degree.

Note that when a graph represents a diagram, vertices represent cut-open curves that are identified. As such, the number of arcs on one vertex must equal the number of arcs on the paired vertex. These are called \textit{weight equations}, as both vertices must be weighted the same. There are only three possible (connected) planar graphs on 4 vertices that satisfy the weight equations: (1) $K_4$ (the tetrahedron) which has no split pairs of stacks and no waves as seen in Figure~\ref{graph1}, (2) a graph with a split pair of stacks but no switchbacks, as seen in Figure~\ref{graph2} and (3) a graph with switchbacks and a split pair of stacks, as seen in Figure~\ref{graph3}.  The graphs are labeled with weights on each edge, where each weight is a nonnegative integer.

\begin{figure}[ht!]
   \begin{center}
  %  \begin{picture}(600, 100)(20,20)
            \includegraphics[scale=.7]{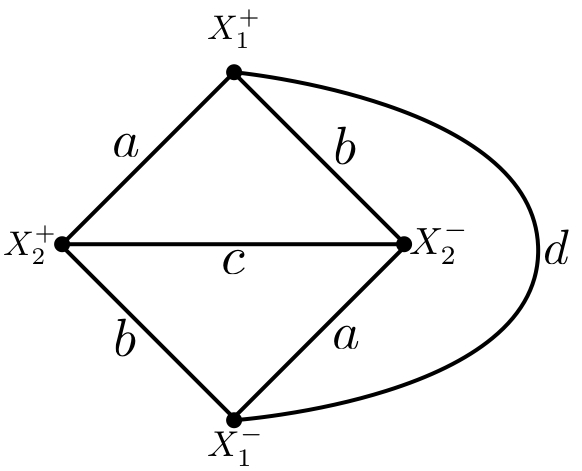}

      %  \end{picture}
        \caption{Complete graph on four vertices}\label{graph1}
    \end{center}
\end{figure}
\begin{figure}[ht!]
   \begin{center}
  %  \begin{picture}(600, 100)(20,20)
            \includegraphics[scale=.7]{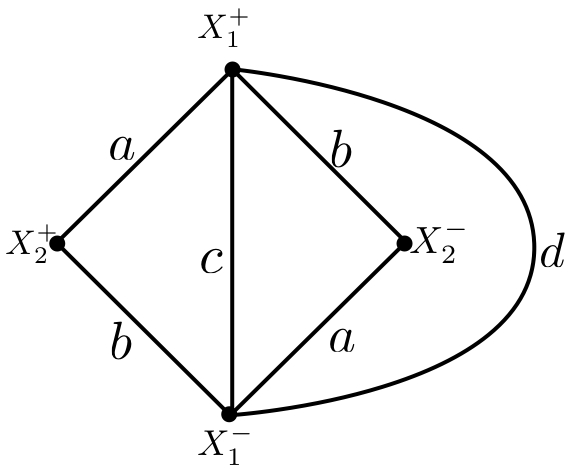}

      %  \end{picture}
        \caption{Graph on four vertices with a split pair of stacks}\label{graph2}
    \end{center}
\end{figure}\begin{figure}[ht!]
   \begin{center}
  %  \begin{picture}(600, 100)(20,20)
            \includegraphics[scale=.7]{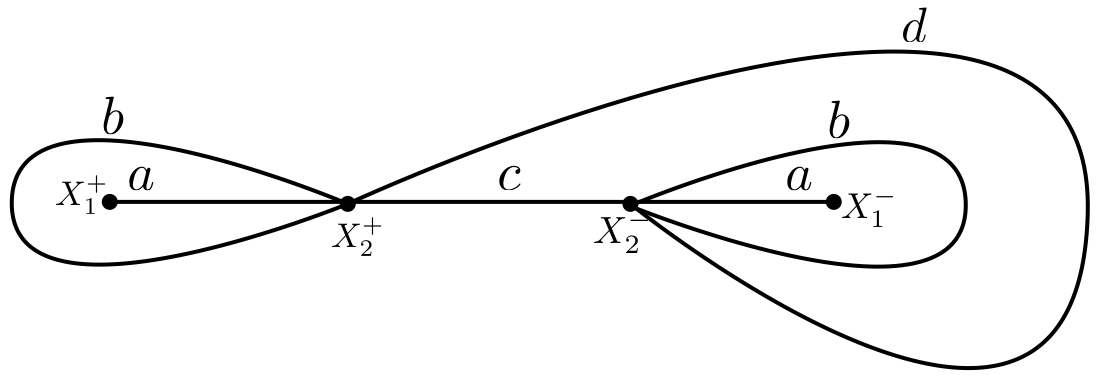}

      %  \end{picture}
        \caption{Graph on four vertices with a split pair of stacks and switchbacks}\label{graph3}
    \end{center}
\end{figure}

As there are only four choices for $A$ (really just two up to inverting through $Z$), we can quickly examine them for each graph. In graphs (1) and (2), all possibilities for $A$ can be realized with $Z$ encircling a stack, both vertices hosting the stack, and nothing else; therefore $Z$ will not introduce switchbacks.

There are two ways to label graph (3), by either placing vertices $X_j^+, X_i^+$ adjacent or $X_j^+, X_i^-$ adjacent.  For now, assume there is an edge between $X_j^+, X_i^+$ as depicted in Figure~\ref{graph3}. Then there is no edge between $X_j^+, X_i^-,$ and performing a Whitehead homeomorphism on this pair will induce a switchback crossing curve $c$. However, in such a case, it would be beneficial to first perform a Whitehead homeomorphism via a wave along the stack with endpoints on $\{X_j^+,X_i^+\}\in A.$ This will remove the switchback, and decrease the geometric degree, putting us in the case of graphs (1) and (2).

Also, the geometric degree of $D(P)$ can be reduced monotonically (Theorem~\ref{WhThm}), meaning the Whitehead homeomorphism that created the switchback from crossing $c$ (twice) would result in a geometric degree that was $2c$ greater than if we had not introduced the switchback. Thus, in the two-generator case, the geometric degree can be reduced monotonically without introducing switchbacks.

\begin{prop}\label{DecGeomDeg}
Let $D$ be a two-meridian $s$-diagram with switchbacks. Then there is a Whitehead homeomorphism that will remove the switchbacks and decrease the geometric degree.
\end{prop}

As a result, we also conclude that the a diagram of lowest geometric degree will not have switchbacks. If we began with an $s$-diagram $D$, then at minimal $deg_G(D)$ we would have no switchbacks and retain an embeddable diagram $D,$ so we have the following corollary.

\begin{cor}\label{minDeg}
Given a two-meridian $s$-diagram $D$, the diagram of minimal geometric degree $D'$ will be an exact $s$-diagram, and $D$ and $D'$ differ by some sequence of Whitehead homeomorphisms.
\end{cor}

\begin{thm}\label{2gSolved}
If we have a two-generator presentation, $P,$ represented by an $s$-diagram, then $P$ is represented by an exact $s$-diagram.
\end{thm}

\begin{proof}
If the $s$-diagram is exact, we are done.

If the $s$-diagram is not exact, then by Proposition~\ref{DecGeomDeg} we can reduce the geometric degree by Whitehead homeomorphisms and at the minimal geometric degree, we get another $P'$ which is represented by an exact $s$-diagram (Corollary~\ref{minDeg}). Turn that process around, and now take the sequence of inverse Whitehead homeomorphisms in the opposite order, each of which does not induce switchbacks.  The sequence of homeomorphisms corresponds to automorphisms that takes $P'$ back to $P.$
\end{proof}

\begin{cor}
The problem of deciding whether an arbitrary presentation with two generators naturally presents the fundamental group of a 3-manifold is solved.
\end{cor}

\section{The Isomorphism Problem for 3-manifold groups}\label{6}

The study of 3-manifolds is still quite lively because the Homeomorphism Problem for 3-manifolds is unresolved. The corresponding problem for 2-manifolds is solved \cite{DH} and for $n$-manifolds ($n\geq 4$) the problem is unsolvable \cite{12, 65}.  It remains open whether this problem is solvable for 3-manifolds and if so, to provide efficient solutions.  We suggest some steps towards this.

If $M_1$ and $M_2$ are homeomorphic, then $\pi_1(M_1)\cong\pi_1(M_2).$ That is, we can consider the Isomorphism Problem (an equally difficult problem) rather than the Homeomorphism Problem.  Even though this problem is difficult, we can still give a partial answer since what we have done has a nice invariant of groups hidden in it.  For any finitely presented group, we can ask what is the smallest genus surface on which a presentation of that group can be written? This is defined for all finitely presented groups, since for every finitely presented group we can determine a finite presentation for that group, and every finite presentation can be realized by a set of curves on a genus $g$ surface. Minimizing over $g(S)$ (or $g(\partial_X M)$ if preferred) gives us our group invariant. Additionally, this process quickly picks out a 3-manifold group when $g(\partial_X M)=0.$ If we can determine that two  fundamental groups are not isomorphic, then we know  that the two manifolds cannot be homeomorphic.

Given $G_1:=\pi_1(M_1)$ and $G_2:=\pi_1(M_2)$ create presentations $P_i$ for $i=1,2.$ We wish to find the lowest genus splitting surfaces which realize these presentations. That is, over all diagrams $(S_i;X,Y)$ in $[D(P_i)],$ what is $\min\{g(S)\}?$ This is only helpful in a limited setting because there are several problems, all related to constructions not being well-defined.

The first problem is that there is not a unique presentation for a group and second, $[D(P_i)]$ is an infinite class of diagrams. Therefore we can in general only give upper bounds to this invariant. However, we have developed methods that will allow us to recognize circumstances in which $\partial_XM=\emptyset$ and $\partial_XM\neq \emptyset,$ and as such it is more practical to use $g(\partial_XM_i)$ as an invariant rather than $g(S_i).$ 
%Note that for an $s$-diagram $g(S)-\beta_0(X)=g(\partial_XM)$ (See Theorem~\ref{BdX}).

Again, we can only determine this in a limited setting, as we only provided sufficient conditions for recognizing whether or not $\partial_XM=\emptyset$.  We can only determine $g(\partial_XM_2)\neq 0$ if we chose a diagram that is adaptable to the methods included in \cite{thesis}.

Another problem is in recognizing diagrams for which $g(\partial_XM_1)=0.$ If $\partial_X M$ is empty, we have an algorithm that would eventually determine this. However, this algorithm is still inefficient until an upper bound on geometric degree is determined, indicating when we should stop searching for such a diagram.

\end{document}